\documentclass[sn-mathphys,Numbered]{sn-jnl}

\usepackage{graphicx}%
\usepackage{multirow}%
\usepackage{amsmath,amssymb,amsfonts}%
\usepackage{amsthm}%
\usepackage{mathrsfs}%
\usepackage[title]{appendix}%
\usepackage{xcolor}%
\usepackage{textcomp}%
\usepackage{manyfoot}%
\usepackage{booktabs}%
\usepackage{algorithm}%
\usepackage{algorithmicx}%
\usepackage{algpseudocode}%
\usepackage{listings}%

\usepackage{hyperref}
\usepackage{enumerate}
\usepackage{subcaption}

\theoremstyle{plain}
\newtheorem{theorem}{Theorem}

\theoremstyle{definition}

\newtheorem{remark}{Remark}

\newcommand{\Lcal}{\mathcal{L}}
\newcommand{\Acal}{\mathcal{A}}
\newcommand{\Dt}{\Delta t}

\newcommand{\Isf}{\mathsf{I}}
\newcommand{\Asf}{\mathsf{A}}
\newcommand{\Ssf}{\mathsf{S}}
\newcommand{\Tsf}{\mathsf{T}}
\newcommand{\Dsf}{\mathsf{D}}
\newcommand{\Lsf}{\mathsf{L}}
\newcommand{\Esf}{\mathsf{E}}
\newcommand{\Vsf}{\mathsf{V}}

\newcommand{\usf}{\mathsf{u}}
\newcommand{\vsf}{\mathsf{v}}
\newcommand{\wsf}{\mathsf{w}}
\newcommand{\gsf}{\mathsf{g}}
\newcommand{\fsf}{\mathsf{f}}
\newcommand{\xsf}{\mathsf{x}}

\graphicspath{{./fig/}}  

\begin{document}

\title[Article Title]{Article Title}


\title[Article Title]{Fast and high-order approximation of parabolic equations using hierarchical direct solvers and implicit Runge-Kutta methods}

\author*[1]{\fnm{Ke} \sur{Chen}}\email{kechen@umd.edu}

\author[2]{\fnm{Daniel} \sur{Appel\"o}}\email{appelo@vt.edu}

\author[3]{\fnm{Tracy} \sur{Babb}}\email{tracy.babb@colorado.edu}

\author[4]{\fnm{Per-Gunnar} \sur{Martinsson}}\email{pgm@oden.utexas.edu}

\affil*[1]{\orgdiv{Department of Mathematics}, \orgname{University of Maryland}, \orgaddress{\street{4176 Campus Drive}, \city{College Park}, \postcode{20742}, \state{MD}, \country{USA}}}

\affil[2]{\orgdiv{Department of Mathematics}, \orgname{Virginia Tech.}, \orgaddress{\street{225 Stanger Street}, \city{Blacksburg}, \postcode{24061}, \state{VA}, \country{USA}}}

\affil[3]{\orgdiv{Applied Mathematics}, \orgname{University of Colorado Boulder}, \orgaddress{\street{2300 Colorado Avenue}, \city{Boulder}, \postcode{80309}, \state{CO}, \country{USA}}}

\affil[4]{\orgdiv{Department of Mathematics}, \orgname{University of Texas at Austin}, \orgaddress{\street{2515 Speedway}, \city{Austin}, \postcode{78712}, \state{TX}, \country{USA}}}


\abstract{A stable and high-order accurate solver for linear and nonlinear parabolic equations is presented. 
	An additive Runge-Kutta method is used for the time stepping, which integrates the linear stiff terms by an explicit singly diagonally implicit Runge-Kutta (ESDIRK) method and the nonlinear terms by an explicit Runge-Kutta (ERK) method. In each time step, the implicit solve is performed by the recently developed Hierarchical Poincar\'e-Steklov (HPS) method. 
	This is a fast direct solver for elliptic equations that decomposes the space domain into a hierarchical tree of subdomains and builds spectral collocation solvers locally on the subdomains. 
	These ideas are naturally combined in the presented method since the singly diagonal coefficient in ESDIRK and a fixed time-step ensures that the coefficient matrix in the implicit solve of HPS remains the same for all time stages. 
	This means that the precomputed inverse can be efficiently reused, leading to a scheme with complexity (in two dimensions) $\mathcal{O}(N^{1.5})$ for the precomputation where the solution operator to the elliptic problems is built, and then $\mathcal{O}(N \log N)$ for the solve in each time step.
	The stability of the method is proved for first order in time and any order in space, and numerical evidence substantiates a claim of stability for a much broader class of time discretization methods. 
	Numerical experiments supporting the accuracy of efficiency of the method in one and two dimensions are presented.}



\maketitle

\section{Introduction}
In this paper we consider numerical methods for solving  parabolic equations of the form
\begin{equation}\label{eqn:general}
	\begin{aligned}
		u_t &= \Lcal u + q + g(u)\,, \quad &\text{in  } (0,T)\times \Omega \\
		u   &= f\,, \quad &\text{on  } (0,T)\times \partial\Omega \\
		u   &= u_0\,, \quad &\text{on  } \{0\}\times \Omega \\
	\end{aligned}
\end{equation}
where $\Lcal$ denotes a general second order elliptic differential operator and $\Omega$ is a bounded domain in $\mathbb{R}^d$.
The function $q = q(t,x)$ denotes an external source and $g=g(t,u)$ denotes a nonlinear term. We consider the case of Dirichlet boundary conditions,  $f$ and denote the initial data by $u_0$. The nonlinear term is assumed to contain derivative operators of degree no greater than one so that the elliptic term $\Lcal u$ dominates $g$, i.e.~the equation is of parabolic type. 
To solve this equation numerically, it is usually preferred to discretize the linear parabolic part by an implicit method to avoid numerical stiffness. In each timestep a linear elliptic problem must be solved and when there are many timesteps it can become advantageous to use a direct solver. This is of course especially true if the complexity (as it is here) of the linear solver is good. In fact, although the pre-computation needed to build a fast direct solver can be more expensive than solving the equation once with an iterative solver, once a solution operator has built, each subsequent solve is very fast. 
In this paper we use an efficient direct elliptic solver called ``Hierarchical Poincar\'e-Steklov(HPS)" coupled with high order Runge-Kutta(RK) time discretization to develop a fast solver to the parabolic equation \eqref{eqn:general}.

The HPS solver is a domain decomposition scheme with spectral collocation discretization on each subdomain. It is drew from the classical nested dissection and multifrontal methods~\cite{davis2006direct,george1973nested} often used for low order finite difference discertizations but differs from these in that it achieves very high order of accuracy of spatial derivatives. 
The HPS method was first proposed in~\cite{martinsson2013direct} for elliptic and Helmholtz equations and later generalized in~\cite{gillman2015spectrally,gillman2014direct,hao2016direct} for general elliptic equation and higher dimensions.
The HPS is a highly efficient direct solver to elliptic equations with high order of convergence~\cite{pgm13,pgmHPS14,martinsson2019fast}. 
These advantages remain when HPS is used for the elliptic solve in combination with implicit Runge-Kutta discretizations of parabolic equations, and preliminary results \cite{babb2018hps,babb2018accelerated} indicated that the resulting scheme is stable and efficient. 
However, challenges remain for proving stability of explicit-implicit schemes. 
For example, there is usually an order barrier for bound preserving higher-order implicit schemes~\cite{gottlieb2011strong} and developing a second order in time for general partial differential equations is still open, though several second order schemes exist for kinetic equations~\cite{hu2018asymptotic}. 
In this manuscript we combine high-order Runge-Kutta schemes with HPS to develop a fast and accurate solver for general parabolic equations. 
Stability is rigorously proven for the case where HPS of any spatial order of accuracy is combined with a first order accurate time discretization (backward Euler). 
That the method remains stable for high order time discretizations as well is substantiated through extensive numerical examples in one and two dimensions.

In Section \ref{sec:hps} we introduce the HPS method for elliptic equations. Section \ref{sec:rkhps} combines HPS with high order Runge-Kutta time discretization and describes the proposed solvers for general parabolic equations. In Section \ref{sec:stable} we investigate the stability of the scheme. 
Numerical examples for one and two dimensional problems are given in Section  \ref{sec:numeric}. 

	\section{The HPS method for time-independent problems}\label{sec:hps}

In this section, we briefly review the HPS method \cite{martinsson2013direct,martinsson2012composite,gillman2015spectrally,gillman2014direct,hao2016direct}
that we use for the spatial discretization;
for additional details, we refer to \cite[Sec.~24--26]{martinsson2019fast}.
For concreteness, let us consider the Dirichlet boundary value problem,
\begin{equation}\label{eqn:elliptic}
	\begin{aligned}
		\Acal u(x) &= g(x) \,, &x\in \Omega, \\
		u(x) &= f(x)\,, & x\in \partial\Omega,
	\end{aligned}
\end{equation}
where $\Acal$ is a general second order elliptic differential operator 
\begin{equation}\label{eqn:operatorA}
	\begin{aligned}
		\Acal u(x) =& -c_{11}(x)\partial_1^2u(x) -2c_{12}(x)\partial_1\partial_2u(x) -c_{22}(x)\partial_2^2u(x) \\
		&+c_1(x)\partial_1 u(x) + c_2(x) \partial_2 u(x) + c(x) u(x)\,.
	\end{aligned}
\end{equation}
Here $g(x)$ is the source term and $f(x)$ the boundary data. The HPS method first partitions the domain $\Omega\subset \mathbb{R}^2$ into a hierarchical tree of subdomains $\Omega^\tau\,, \tau = 1,\ldots,N$ and uses spectral collocation methods to discretize the local elliptic operators $\Acal^\tau$. After the discretization, local solution operators and Dirichlet to Neumann (DtN) maps restricted in each subdomain are built at the bottom level, where only small matrix inversion is involved. These local operators are then used to build the global direct solvers via a hierarchical merge from the bottom level to the top level. Once the global solver is constructed, HPS can can solve \eqref{eqn:elliptic} for multiple sources and boundary data via a top-to-bottom sweep. In the implementation of HPS, we assume the mixed term $c_{12}(x)=0$ so one can ignore the corner points in Figure \ref{fig:domain}. For problems involving a nonzero mixed term $c_{12}(x)$, one can also ignore the corner points in HPS, and then apply extrapolation methods to estimate the values on corner points. More details can be found in Chapter 24 of \cite{martinsson2019fast}.

\subsection{Discretization and numerical schemes}\label{sec:discretization}
For simplicity we assume the domain $\Omega$ is rectangular. The domain is first partitioned into two children subdomains and then each children subdomain is further partitioned in a similar manner. This hierarchical partition will stop until the leaf children reach the preset square size and consequently a hierarchical tree of domains are formed. Eventually the domain is partitioned into $n_1\times n_2$ squares of the same size and each square is discretized by $p\times p$ Chebyshev nodes, see Figure \ref{fig:domain}.
\begin{figure}[htp]
	\centering
	\includegraphics[width=1.0\textwidth]{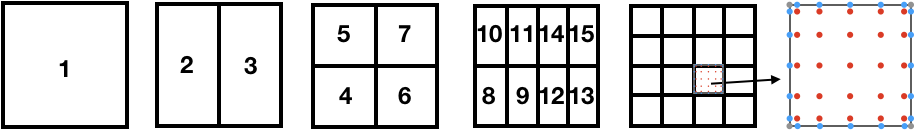}
	\caption{The domain $\Omega$ is hierarchically halved and eventually is partitioned into squares. Then each square is discretized with $p\times p$ Chebyshev nodes. }
	\label{fig:domain}
\end{figure}
We denote all collocation points by $\xsf = \{x_i\}_{i=1}^N$ and define the discretized solution $\usf = [u(x_i)]_{i=1}^N$ as the solution vector with collated values. The Chebyshev nodes on each square is classified into three groups: the interior nodes (marked in red), the boundary edge nodes (marked in blue) and the corner nodes (marked in gray). It turns out that the gray nodes does not contribute to any spectral derivatives to other nodes when there are no mixed derivatives in \eqref{eqn:operatorA}, so we dropped them off in the collocation points $\xsf $. After excluding all the corner points, the total number of points in the grid equals
\[
N = (p-2)(pn_1 n_2 + n_1 + n_2) \approx p^2 n_1 n_2\,.
\]
The discretization procedure above would introduce a sparse $N\times N$ matrix $\Asf$, 
and the value of $\Asf(i,:)\usf$ may have different values depending on the type of point $x_i$
\begin{equation}\label{eqn:sparse_A}
	\Asf(i,:) \usf \approx \begin{cases}
		[\Acal u](x_i), &\text{  for any interior points (marked in red)}\,, \\
		0, &\text{  for any edge point that is }\textit{not}\text{ on }\partial \Omega\,, \\
		\pm \frac{\partial u}{\partial n},& \text{  for any edge point that is on }\partial\Omega \,.
	\end{cases}
\end{equation}

In the HPS solver, this big sparse matrix $\Asf$ is not explicitly formed but instead the matrix-vector application of $\Asf^{-1}$ is constructed via a hierarchical sweep through all subdomains.
To explain the details, we first restrict our scope to a local subdomain $\Omega^\tau$. Denote the interior nodes index in a square $\Omega^\tau$ as $J^\tau_{i}$ and the boundary edge nodes index as $J^\tau_{b}$, the operator $\Acal$ can be locally discretized as a spectral differential matrix and likewise equation \eqref{eqn:elliptic} can be discretized in the following form for any leaf node $\tau$:
\begin{equation}\label{eqn:discretize}
	\begin{aligned}
		\begin{bmatrix}
			\Asf^\tau_{i,i} & \Asf^\tau_{i,b} \\
		\end{bmatrix}
		\begin{bmatrix}
			\usf^\tau_{i}\\
			\usf^\tau_{b}
		\end{bmatrix}
		&= \gsf^\tau_{i}, \\
		\usf^\tau_{b} &= \fsf^\tau_{b},
	\end{aligned}
\end{equation}
where the subscript denotes the values on the corresponding collocation nodes. For example, $\Asf^\tau_{i,b} = \Asf^\tau(J^\tau_{i},J^\tau_{b})$ and $\gsf^\tau_i = \gsf(J^\tau_{i})$. The resulting local problem is of small size $p$ and thus it is easy to construct the solution operator $S^\tau$ that maps the boundary data to the interior solution
\[
\usf_i^\tau = \Ssf^\tau \usf_b^\tau \,.
\]
Additionally, we can build the DtN operator $\Tsf^\tau$ that maps the boundary data $\usf^\tau_b$ to a vector $v^\tau$ consisting of the boundary fluxes 
\[
\vsf^\tau = \Tsf^\tau \usf_b^\tau \,.
\]
More precisely, if the collocation point $x_i$ is on a vertical edge, then $\vsf^\tau(i) \approx \frac{\partial \usf}{\partial x_1}(x_i)$ and if $x_i$ is on a horizontal edge, then $\vsf^\tau(i) \approx  \frac{\partial u}{\partial x_2}(x_i)$. These two operators can be computed directly from the local spectral differentiation matrix. For example, in the absence of source term $\gsf_i$, the solution operator $\Ssf^\tau$ and DtN operators are
\begin{equation}
	\label{eqn:leaf}
	\Ssf^\tau = - \Asf^{-1}_{i,i} \Asf_{i,b} \quad \text{and}\quad \Tsf^\tau = \Dsf \Ssf^\tau,
\end{equation}
where $\Dsf$ consists of spectral differentiation operators on edge nodes corresponding to $\frac{\partial}{\partial x_1}$ and $\frac{\partial}{\partial x_2}$ respectively. 


The full hierarchical HPS solver consists of two stages: a build stage that sweeps from leaf squares to its parent, and a solve stage that pass through the tree starting from the root to its leaves. In the building stage, the solution operator $\Ssf^\tau$ and DtN operator $\Tsf^\tau$ are built for each subdomain $\tau$ from leaves to roots. For leaf subdomains, they can be built directly by using equation \eqref{eqn:leaf}. 
For a parent subdomain $\tau$ with children subdomains $\alpha$ and $\beta$, $\Ssf^\tau$ and $\Tsf^\tau$ can be built by ``merging" the DtN operators $\Tsf^\alpha$ and $\Tsf^\beta$ of the children subdomains $\alpha$ and $\beta$. These process is done by Schur complements of the linear system \eqref{eqn:discretize} with respect to that of its children via the continuity of solution and fluxes on the interface. See more details in Chapter 25 of \cite{martinsson2019fast}.
For completeness, we briefly discuss below how to merge the local operators on children domains into the local operators on the parent domain. We first partition the boundary points on $\partial\Omega^\alpha$ and $\partial\Omega^\beta$ into the three sets: $J_1$, the boundary points on $\partial \Omega_\alpha/\partial \Omega_\beta$; $J_2$, the boundary points on $\partial \Omega_\beta/\partial \Omega_\alpha$; and $J_3$, the boundary points on both $\partial\Omega_\alpha$ and $\partial\Omega_\beta$ but not on $\partial\Omega_\tau$.
See Figure \ref{fig:merge} for an illustration of the boundary points. The goal of the merge process is to construct the solution operator $\Ssf^\tau$ and DtN operator $\Tsf^\tau$ on the parent node $\tau$ from the local operators $S^\alpha,T^\alpha$ and $S^\beta,T^\beta$ on its children nodes. These children local operators can be partitioned as the following.
\begin{equation}\label{eqn:local_operators_children}
\begin{bmatrix}
    \vsf_1 \\ \vsf_3
\end{bmatrix}
=
\begin{bmatrix}
    \Tsf^\alpha_{1,1} & \Tsf^\alpha_{1,3}\\
    \Tsf^\alpha_{3,1} & \Tsf^\alpha_{3,3}\\
\end{bmatrix}
\begin{bmatrix}
    \usf_1 \\ \usf_3
\end{bmatrix} \,, \quad
\begin{bmatrix}
    \vsf_2 \\ \vsf_3
\end{bmatrix}
=
\begin{bmatrix}
    \Tsf^\alpha_{2,2} & \Tsf^\alpha_{2,3}\\
    \Tsf^\alpha_{3,2} & \Tsf^\alpha_{3,3}\\
\end{bmatrix}
\begin{bmatrix}
    \usf_2 \\ \usf_3
\end{bmatrix}  \,.
\end{equation}
By definition, the local operators on the parent nodes can also partitioned analogously.
\begin{equation}\label{eqn:local_operators_parent}
    \usf_3 = \Ssf^\tau 
    \begin{bmatrix}
    \usf_1 \\ \usf_2
\end{bmatrix}\,, \quad
\begin{bmatrix}
    \vsf_1 \\ \vsf_2
\end{bmatrix}
=\Tsf^\tau 
\begin{bmatrix}
    \usf_1 \\ \usf_2
\end{bmatrix}
\end{equation}
Combining \eqref{eqn:local_operators_children} and \eqref{eqn:local_operators_parent}, we can solve for $\Ssf^\tau$ and $\Tsf^\tau$ with $\Tsf^\alpha$ and $\Tsf^\beta$ as the following.
\[
\Ssf^\tau = \left(\Tsf^\alpha_{3,3} - \Tsf^\beta_{3,3} \right)^{-1} \left[-\Tsf^\alpha_{3,1} \ |\ \Tsf^\beta_{3,2} \right] \,, \quad 
\Tsf^\tau = 
\begin{bmatrix}
    \Tsf^\alpha_{1,1} & 0 \\
    0 & \Tsf^\beta_{2,2}
\end{bmatrix} + 
\begin{bmatrix}
    \Tsf^\alpha_{1,3} \\
    \Tsf^\beta_{2,3}
\end{bmatrix}
\Ssf^\tau \,.
\]

\begin{figure}[htp]
	\centering
	\includegraphics[width=0.6\textwidth]{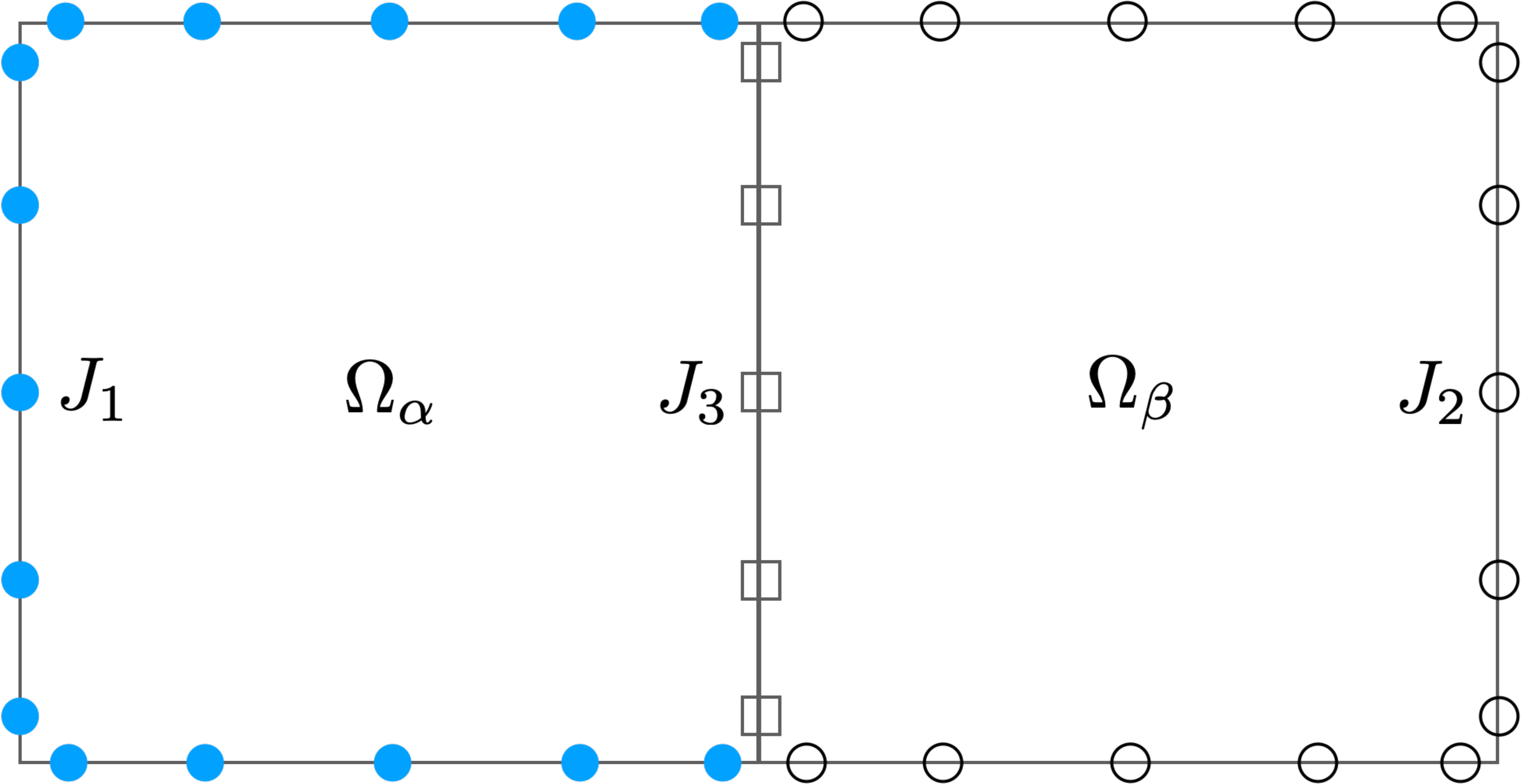}
	\caption{An illustration of boundary points on the parent domain $\Omega^\tau = \Omega^\alpha \cup \Omega^\beta$.}
	\label{fig:merge}
\end{figure}

In the solve stage, one starts from the root domain $\Omega$ and iteratively uses the solution operator to map boundary data to the interface of its children subdomains. This process would pass boundary information from parent domains to its children and consequently provide boundary data for all leaf subdomains. Then with the solution operator $\Ssf^\tau$ for all leaf subdomains, one can easily construct local solutions and glue them together into a global solution.
We summarized the solve stage of HPS in the following algorithm \ref{alg:solve} when no external force is present. 
\begin{algorithm}
	\caption{HPS solve stage with no body load.}
	\label{alg:solve}
	\begin{algorithmic}[1]
		\State $\usf(k) = f(x_k)$ for all $k\in J^1_b$.
		\For{$\tau = 1,2,\ldots,n$}
		\State $\usf(J^\tau_i) = \Ssf^\tau \usf(J^\tau_b)$.
		\EndFor
	\end{algorithmic}
\end{algorithm}

The time complexity to build all local matrices for leaves in 2D is about $(p^2)^3 = p^6$ because there are $p^2$ points for each leaf subdomains. As there are about $N/p^2$ leaf subdomains, the total cost to process all leaf subdomains is about $(p^6)(N/p^2) = p^4N$. For a parent subdomain $\tau$ at level $l$, the cost of merging process is about $2^{-2l}N^{1.5}$. Because there are $2^l$ subdomains at level $l$, the total cost at level $l$ is about $2^{-l}N^{1.5}$ and thus the total build cost in 2D is $\sum_{l=0}^{L} 2^{-l }N^{1.5} \approx N^{1.5}$. For the solve stage, the cost of applying the solution operator $S^\tau$ is about $2^{-l}N$ so the total cost of the solve stage in 2D is $\sum_{l=0}^{L-1} 2^l 2^{-l}N \approx N\log N$. We numerically verify the solve stage time complexity in Section \ref{sec:compareFEM}.

\section{The RKHPS method for time-dependent problems}\label{sec:rkhps}

We now introduce the Runge-Kutta Hierarchical Poincar\'e-Steklov (RKHPS) schemes for the  general parabolic equation \eqref{eqn:general}. These  schemes combine high order Runge-Kutta discretization and high order HPS schemes, thus they enjoy high accuracy, stability and efficiency. 
As mentioned above the HPS scheme is most efficient if the matrix $\Asf$ does not change throughout the computation and it is therefore natural to consider Explictly, Singly Diagonally Implicit Runge-Kutta (ESDIRK) methods.
\subsection{Time discretization}
In general, the right hand side of \eqref{eqn:general} can be split into a stiff part $\Lcal u$ and a non-stiff part $q + g(u)$.
The natural choice is to adopt an implicit-explicit RK (IMEX-RK) method, in which $\Lcal u$ is discretized by a method defined through an implicit Butcher table $A,b$ and $c$ whereas $g(u)$ is treated by a method defined by an explicit Butcher table $\hat{A},\hat{b}$ and $\hat{c}$. In particular, we use the ESDIRK scheme for the stiff term $\Lcal u$. These methods usually have an explicit first stage and put the same constant on the diagonal entries $a_{ii}$. 
Such structure in the Butcher table $A$ allows to have the same sparse matrix $\Asf$ in \eqref{eqn:sparse_A} in all iterations. Therefore, we can build a HPS solver once in the offline and apply the same HPS solve for all iterations.
The nonstiff term is dealt with an Explict Runge-Kutta (ERK) method, which only has non-zero entries in the lower diagonal part in the Butcher table. See Table \ref{tbl:butcher}. In particular, we used the tableau data \textbf{ARK4(3)6L[2]SA} and \textbf{ARK5(4)8L[2]SA} from the additive Runge-Kutta (ARK) methods by Carpenter and Kennedy ~\cite{kennedy2003additive}. It is assumed that the time step coefficients $c_j$ and $\hat{c}_j$ are the same in both ESDIRK and ERK method.
\begin{table}[htp]
\parbox{.45\linewidth}{
	\centering
	\begin{tabular}{c|c c c c c}
		$0$ & 0 &0 			& 				 &  $\ldots$  & 0\\
		$c_2$ &$a_{21}$ &$\gamma$ & 0  				  &  $\ldots$& 0\\
		$\vdots$& $\vdots$ & $\ddots$ & $\ddots$ &$\ddots$ & $\vdots$\\
		$c_{s-1}$ &$a_{s-1,1}$ &$a_{s-1,2}$ &$\ldots$&$\gamma$  	  & 0 \\
		$1$ &$b_1$ &$b_2$ &$\ldots$ & $b_{s-1}$ 	  & $\gamma$ \\	\hline
		&$b_1$ &$b_2$ & $\ldots$ & $b_{s-1}$ 	  & $\gamma$ 
	\end{tabular}
}
\hfill
\parbox{.45\linewidth}{
        \centering
	\begin{tabular}{c|c c c c c}
		$0$ & 0 &0 			& 				 &  $\ldots$  & 0\\
		$\hat{c}_2$ &$\hat{a}_{21}$ &0& 0  				  &  $\ldots$& 0\\
		$\vdots$& $\vdots$ & $\ddots$ & $\ddots$ &$\ddots$ & $\vdots$\\
		$\hat{c}_{s-1}$ &$\hat{a}_{s-1,1}$ &$\hat{a}_{s-1,2}$ &$\ldots$&$0$  	  & 0 \\
		$1$ &$\hat{a}_{s,1}$ &$\hat{a}_{s,2}$ &$\ldots$ & $\hat{a}_{s,s-1}$ 	  & 0 \\	\hline
		&$\hat{b}_1$ &$\hat{b}_2$ &$\ldots$ & $\hat{b}_{s-1}$ 	  & $\hat{b}_s$
	\end{tabular}
 }
	\caption{Butcher tables of Runge-Kutta methods. \textbf{Left:}  ESDIRK. \textbf{Right}: ERK}
	\label{tbl:butcher}
\end{table}


The general RK method can be usually formulated in two ways: the stage formulation or the slope formulation. These two formulations are algebraically equivalent for ODE systems but not necessarily equivalent for PDEs. In this manuscript, we implement and compare both formulations. 
The stage formulation consists of $s$ intermediate stage solves:
\begin{equation}\label{eqn:stage}
	u^n_i = u^n + \Dt \sum_{j=1}^{i} a_{ij} \Lcal  u_j^n + \Dt \sum_{j=1}^{i-1} \hat{a}_{ij} \left( q^n_j + g^n_j \right) \,,\ \ i = 1,\ldots,s,
\end{equation}
where $t^n_j = t^n + c_i \Dt$, $q^n_j = q(t^n_j,x)$ and $g^n_j=g(u^n_j)$. With a ESDIRK method shown in Table \ref{tbl:butcher}, we can assume that $a_{ii} = \gamma, i=2,\ldots,s$ and rewrite equation \eqref{eqn:stage} as the following
\begin{equation}\label{eqn:stage2}
	\begin{aligned}
		u^n_1 &= u^n \\
		(\Isf - \Dt \gamma \Lcal ) u^n_i &=  u^n + \Dt \sum_{j=1}^{i-1} a_{ij} \Lcal  u_j^n + \Dt \sum_{j=1}^{i-1} \hat{a}_{ij} \left( q^n_j + g^n_j \right)\,,\ \ i = 2,\ldots,s.
	\end{aligned}
\end{equation}
It is clear that the first stage is explicit and for other stages one needs to invert an elliptic operator $(\Isf-\Dt \gamma \Lcal)$. The equations above hold in the interior of $\Omega$ and are equipped with boundary conditions 
\begin{equation}
	u^n_i = f(t^n_i) \,, \quad \text{on  } \partial\Omega \,, \ \ i=1,\ldots,s.
\end{equation}
In comparison, the slope formulations, if there is no nonlinear term $g=0$, is composed of multiple stages where slope variables $k^n_i$ are calculated as follows
\begin{equation}\label{eqn:slope0}
	\begin{aligned}
		k^n_1 & = \Lcal u^n + q^n_j,\\
		(\Isf -\Dt \gamma \Lcal )k^n_i & = \Lcal u^n + \Dt \sum_{j=1}^{i-1} a_{ij} \Lcal k^n_j \,, \ \ i = 2,\ldots,s.
	\end{aligned}
\end{equation}
The semi-discretization
above need to be augmented with suitable boundary conditions. 

	Let $E$ be the diagonal matrix that is 1 at boundary DOF and zero everywhere else. Suppose we want to solve the PDE $v_t = v_{xx} + F^{\rm E}(v)$, with boundary conditions $v = v_{\rm BC}(t)$ using a semi-discretization
	\[
	u_t = D_2 u + \tau E (u - v_{\rm BC}(t)) + F^{\rm E}(u). 
	\]
	Here the boundary conditions are enforced weakly by the penalty term. Denoting $F^{\rm I}(u) = D_2 u + \tau E (u - v_{\rm BC}(t))$ we have that 
	\[
	F^{\rm I}(u) = u_t - F^{\rm E}(u).
	\]
	We consider the first stage in an IMEX method. Given the current solution $u_{n-1}$ it is:
	\begin{enumerate}
		\item Set $K^{\rm E}_1 = F^{\rm E} (u_{n-1})$.
		\item Solve $$
		K_1^{\rm I} = F^{\rm I}(u_1) \equiv F^{\rm I}(u_{n-1} + \Delta t A^{\rm I}_{1,1} K^{\rm I}_1 + \Delta t A^{\rm E}_{2,1} K^{\rm E}_1).$$
		\item Using the explicit expression for $F^{\rm I}$ we only use the penalty term for $K^{\rm I}_1$ and find
		\[
		K_1^{\rm I} = D_2 u_1 + \Delta t A^{\rm I}_{1,1} D_2 K^{\rm I}_1  + \tau E (K^{\rm I}_1 - R)+ \Delta t A^{\rm E}_{2,1} D_2 K^{\rm E}_1.
		\]
		Here we can use the PDE, $F^{\rm I}(u) = u_t - F^{\rm E}(u)$, to find 
		\[
		R = \frac{d v_{\rm BC}(t)}{dt} - E F^{\rm E}(u_i). 
		\]
		As we now have additional unknowns on the boundary from $E F^{\rm E}(u_i)$ we must add the equations 
		$$
		E(K^{E}_{2} - F^{\rm E}(u_i)) = 0.
		$$
	\end{enumerate}
	   A natural choice when $q=0$ is to set $k^n_i = u_t(t^n_i)$ at the boundary as the slope $k^n_i$ can be interpreted as an approximation to the time derivatives of the solution. At the end, the one step approximation $u^{n+1}$ can be calculated explicitly by assembling the slope variables $k^n_i$
	\[
	u^{n+1} = u^n + \Dt \sum_{j=1}^n b_j k^n_j.
	\]
	When the nonlinear term $g$ is present, another explicit slope variable $l^n_i$ is introduced. The intermediate slope variables $k^n_i$ and $l^n_i$ are computed as follows:
	\begin{equation}\label{eqn:slope}
		\begin{aligned}
			k^n_1 & = \Lcal u^n + q^n_j,\\
			l^n_1 &=  g^n_j,\\
			(\Isf -\Dt \gamma \Lcal )k^n_i & = \Lcal u^n + \Dt \sum_{j=1}^{i-1} a_{ij} \Lcal k^n_j + \Dt \sum_{j=1}^{i-1} \hat{a}_{ij}   \Lcal l^n_j + q^n_i\,, \ \ i = 2,\ldots,s,\\
			l^n_i &= g(u^n+\Dt \sum_{j=1}^{i}a_{ij}k^n_j + \Dt \sum_{j=1}^{i-1} \hat{a}_{ij} l^n_j ) \,, \ \ i=2,\ldots,s.
		\end{aligned}
	\end{equation}
	Similarly, the one step approximation can be assembled as
	\[
	u^{n+1} = u_n + \Dt \sum_{j=1}^n b_j k^n_j + \Dt \sum_{j=1}^n \hat{b}_j l^n_j.
	\]
	However, a major challenge is to design suitable boundary conditions for the intermediate slopes $k^n_i$ and $l^n_i$.
	As there are no clear interpretation of these individual slope variables on the boundary, it is only possible to assign suitable boundary conditions for limited situations. For instance, zero boundary conditions can be assigned for both $k^n_i$ and $l^n_i$ if BC of the PDE is time-independent.

	\subsubsection{Implicit and explicit computation}
	Notice that in either the stage or slope formulations above, there are two types of equations that need to be solved, i.e. implicit elliptic equations in the form
	\[
	(\Isf - \Dt \gamma \Lcal) u^\tau  = \text{known RHS} \quad \text{on  } \Omega^\tau  \,,
	\] 
	or explicit equations in the form
	\[
	u = \text{known RHS}\quad \text{on  } \Omega^\tau  \,,
	\]
	where $u$ denotes the stage variables $u^n_i$ or slope variables $k^n_i,l^n_i$ respectively. Both the implicit and explicit equations are equipped with Dirichlet boundary conditions and hold on all subdomains. As the identity operator can be interpreted as an elliptic operator, both the implicit and explicit equations can be solved on the hierarchical tree by following standard HPS methods. However, for explicit equations alone, one can explicitly form the global linear sparse system and directly solve it without using HPS method. We have tested and compare both implementations in the arXiv report and found that there are only machine precision difference between directly solving sparse system and using HPS method.
	
	\subsubsection{Boundary conditions}
	As the unknowns $u^n_i$, $k^n_i$ and $l^n_i$ are approximating the solutions and the slope variables respectively, we need to assign boundary conditions in different ways in different formulations correspondingly. For the stage formulation, a natural choice is to assign $u^n_i = u(x,t^n+c_i\Dt)$ for $x\in \partial \Gamma$. However, such treatment may suffer from order reduction \cite{rosales2017order} as shown in the numerical results. For the slope formulation, if only $k^n_i$ is present, the natural choice is to assign $k^n_i = u_t(x,t^n+c_i\Dt)$ for $x\in \partial \Gamma$ because $k^n_i$ is approximating the time derivative of the solutions. In the case when both the implicit slope $k^n_i$ and explicit slope $l^n_i$ are both present, there is no clear relation between them and the solution, thus only limited cases are applicable. For example, if equation \eqref{eqn:general} is equipped with time-independent BC, then zero boundary Dirichlet conditions can be assigned to the slope variables.
	
	To deal with Neumann and Robin boundary conditions, the HPS method uses the pre-computed DtN operators to map them to Dirichlet boundary conditions. More details can be found in \cite{babb2018accelerated}.
	
	\subsubsection{Penalization in the slope formulation}
	The HPS method inherently enforces continuity of the fluxes across the interfaces of children subdomains. Such feature, however in the slope formulation, does not guarantee that the solution has a continuous flux across the interface. As a trivial example, if only $k^n_i$ is present, then the updating formula
	\[
	u^{n+1} = u^n + \Dt \sum_{j=1}^{s} b_j k^n_j \,,
	\]
	will pass any flux mismatch from the previous solution $u^n$ to the next step.
	
	An easy fix of this can be projecting the solution into the continuous flux space. However, for greater generality, we enforce an penalization on flux jump in the slope formulation. That is, the zero flux jump condition in HPS
	\[
	[[Tk + h^k]] = 0 \,,
	\]
	is replaced by a penalized version
	\[
	[[ Tk + h^k - \Dt^{-1} h^u]] = 0 \,,
	\]
	where $Tk$ denotes the derivative from the homogeneous part, $h^k$ denotes the flux of particular slope and $h^u$ denotes the flux of the solution $u$. This new penalized condition modifies the merging process by adding an extra term.

	\section{Stability of RKHPS}\label{sec:stable}
	In this section, we seek to shed light on the stability properties of RKHPS. We establish that the time-stepping map $u^n\rightarrow u^{n+1}$ is stable for the particular case of the heat equation discretized with HPS in space, and backwards Euler in time. For higher order discretization, analysis appears to be challenging, but we present numerical evidence that point strongly towards the conclusion that RKHPS is stable up to order five.
	\subsection{Eigenvalues of the local differentiation matrix}
	The HPS uses spectral collocation method to discretize each subdomain and the corresponding local differentiation matrix are approximating the second order elliptic operator whose eigenvalues are negative. In particular, the eigenvalues of the continuous second derivative with zero boundary conditions are defined as
	\[
	\begin{aligned}
		D^2 u(x) = \lambda u(x)\,, &\quad -1\leq x \leq 1 \\
		u(\pm 1) = 0\,.&
	\end{aligned}
	\]
	The eigenvalues of this continuous problem are known to be $\lambda_k = -\frac{k^2 \pi^2}{4}$. The Chebychev collocation methods considers $u$ as $(p-1)$-th order polynomials such that the above equation holds at Chebyshev points $x_j = \cos(\frac{j\pi}{p-1})$. Such discretization yields the spectral differentiation matrix, which approximates the first $\frac{2}{\pi}$ portion of the eigenvalues very well but there is an $\mathcal{O}(p^4)$ error for the remaining eigenvalues. It is shown in \cite{GottliebLustman,weideman1988eigenvalues} that the eigenvalues of spectral differentiation matrix are real, negative and distinct. 
	
	\subsection{Stability of RKHPS}
	We now prove the stability of RKHPS for linear heat equation in the following form:
	\begin{equation}\label{eqn:heat}
		\begin{cases}
			u_t = \Delta u + f \,, \quad &x\in \Omega, \\
			u   = g\,,\quad & x\in \Gamma, 
		\end{cases}
	\end{equation}
	where $f$ is the external force. 
	\begin{theorem}
		In dimension two, the eigenvalues of the time-stepping map $\mathsf{M}^n: \usf^n\mapsto \usf^{n+1}$ of backward Euler HPS method for heat equation \eqref{eqn:heat} has modulus bounded by $1$ for any time step size $h$. In particular, backward Euler is Lax-Richtmyer stable.
	\end{theorem}
	\begin{proof}
		For stability proof of \eqref{eqn:heat}, it suffices to assume $f = g= 0$. The Backward Euler time discretization for a certain Runge-Kutta scheme yields the following semi-continuous PDE:
		\begin{equation}
			\begin{cases}
				(\Isf+h \Delta) u^{n+1} =  u^n \,,\quad & x\in \Omega, \\
				u^{n+1} = 0 \,, \quad & x\in \Gamma. 
			\end{cases} \,
		\end{equation}	
		In HPS method of the above semi-continous equation, the domain $\Omega$ is decomposed into a hierarchical sequence of subdomains:
		\begin{equation}\label{eqn:localEqn}
			\begin{cases}
				(\Isf+h \Delta) u^{n+1,\tau} =  u^n \,,\quad & x\in \Omega^\tau, \\
				u^{n+1,\tau} = f^\tau \,, \quad & x\in \Gamma^\tau, 
			\end{cases} 
		\end{equation}	
		where $f^\tau$ is the unknown boundary condition over each subdomain.
		Consider the local problem \eqref{eqn:localEqn} with fixed $\tau$, we decompose the local solutions  into particular solution $w^{n+1,\tau}$ and homogeneous solution $\phi^{n+1,\tau}$,
		\[
		u^{n+1,\tau} = w^{n+1,\tau} + \phi^{n+1,\tau} \,.
		\]
		
		They satisfy the following equations respectively. 
		\begin{equation}\label{eqn:decomp}
			\begin{cases}
				(\Isf+h \Delta) w^{\tau} = u^n \,,\quad & x\in \Omega^\tau, \\
				w^{\tau} = 0 \,, \quad & x\in \Gamma^\tau, 
			\end{cases} 
			\quad \text{  and  } \quad 
			\begin{cases}
				(\Isf+h \Delta) \phi^{\tau} =  u^n \,,\quad & x\in \Omega^\tau, \\
				\phi^{\tau} = f^\tau \,, \quad & x\in \Gamma^\tau. 
			\end{cases} 
		\end{equation}
		We omit the supscript $^{n+1}$ in above equations and the proof below when no confusion occurs. 
		In the HPS method, each leaf node $\tau$ is discretized with $p\times p$ Chebyshev points. We denote the indices of all Chebyshev points in $\Omega^\tau$ as $J^\tau$, the interior points as $J^\tau_i$ and the boundary points as $J^\tau_b$. Then the operator $(\Isf + h\Delta)$ is discretized as the $2$nd order Chebyshev differentiation matrix $\Lsf$ and globally it is discretized as a large sparse matrix $\Asf$. (see details in section \ref{sec:discretization}).
		
		Now consider a parent node $\tau$ with children $\alpha$ and $\beta$, equation \eqref{eqn:decomp} implies that the particular solution on the left children must satisfy that
		\[
		\Asf(J^\alpha_i,J^\alpha) \wsf^\alpha = \usf^n(J^\alpha)\,.
		\] 
		Similarly, equation \eqref{eqn:localEqn} implies that
		\[
		\Asf(J^\alpha_i,J^\tau) \wsf^\tau = \usf^n(J^\alpha)\,.
		\]
		Therefore, we must have
		\[
		\Asf(J^\alpha_i,J^\alpha) \wsf^\alpha = \Asf(J^\alpha_i,J^\tau) \wsf^\tau = \Asf(J^\alpha_i,J^\alpha) \wsf^\tau(J^\alpha),
		\]
		where the last equality holds because of the sparsity of matrix $\Asf$. In fact, the values of $\Asf(J^\alpha_i,J^\tau) \wsf^\tau$ depends only on the nodal points within $\Omega^\alpha$ rather than that in $\Omega^\tau$.
		Decomposing the index $J^\alpha = [J^\alpha_i,J^\alpha_b]$, we can rewrite the above equation as
		\[
		\Asf(J^\alpha_i,J^\alpha_i) \wsf^\alpha(J^\alpha_i)  = \Asf(J^\alpha_i,J^\alpha_i) \wsf^\tau(J^\alpha_i) + \Asf(J^\alpha_i,J^\alpha_b) \wsf^\tau(J^\alpha_b),
		\]
		where we used the fact that $\wsf^\alpha(J^\alpha_b) = 0$. Invert the square matrix $\Asf(J^\alpha_i,J^\alpha_i)$, we have
		\[
		\wsf^\alpha(J^\alpha_i)  = \wsf^\tau(J^\alpha_i) + \left(\Asf(J^\alpha_i,J^\alpha_i)\right)^{-1}\Asf(J^\alpha_i,J^\alpha_b) \wsf^\tau(J^\alpha_b).
		\]
		Analogously for child $\beta$, we have
		\[
		\wsf^\beta(J^\beta_i)  = \wsf^\tau(J^\beta_i) + \left(\Asf(J^\beta_i,J^\beta_i)\right)^{-1}\Asf(J^\beta_i,J^\beta_b) \wsf^\tau(J^\beta_b).
		\]
		\begin{figure}
			\centering
			\includegraphics[width=0.6\textwidth]{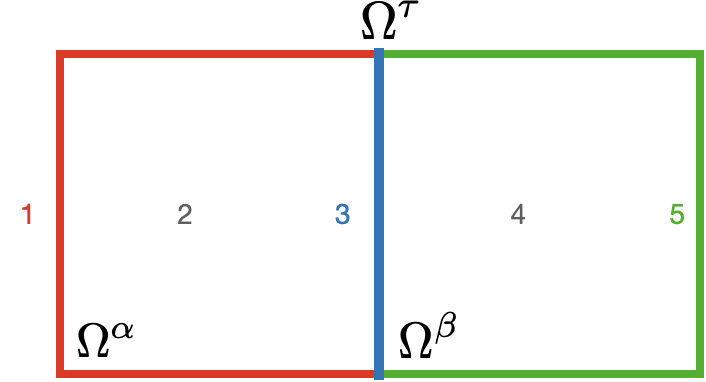}
			\caption{The vector $J^\tau$ is partitioned into five blocks.}
			\label{fig:indexDecomp}
		\end{figure}
		To illustrate the structure of the above two equations, we partition the indices $J^\tau$ into five blocks (see an illustration in Figure \ref{fig:indexDecomp})
		\begin{equation}\label{eqn:parentrelation}
			\begin{bmatrix}
				\wsf^\alpha(J^\alpha_i) \\
				\wsf^\beta(J^\beta_i)
			\end{bmatrix}
			= \begin{bmatrix}
				\Tsf_{11} & \mathbf{I} & \Tsf_{13} & 0 & 0 \\
				0&0&\Tsf_{23} &\mathbf{I} & \Tsf_{25}
			\end{bmatrix}
			\wsf^\tau 
			:= \Tsf \wsf^\tau \,,
		\end{equation}
		where $\Tsf_{11},\Tsf_{13}$ are submatrices of $\left(\Asf(J^\alpha_i,J^\alpha_i)\right)^{-1}\Asf(J^\alpha_i,J^\alpha_b)$ and analogously $T_{23},T_{25}$ are submatrices of  $\left(\Asf(J^\beta_i,J^\beta_i)\right)^{-1}\Asf(J^\beta_i,J^\beta_b)$. 
		Such structure of matrix $\Tsf$ guarantees that itself has singular values bounded below by $1$. In fact, we have
		\[
		\Tsf\Tsf^\top - \Isf= \begin{bmatrix}
			\Tsf_{11}\Tsf_{11}^\top + \Tsf_{13}\Tsf_{13}^\top &  \Tsf_{13}\Tsf_{23}^\top \\
			\Tsf_{23}\Tsf_{13}^\top &  \Tsf_{23}\Tsf_{23}^\top + \Tsf_{25}\Tsf_{25}^\top 
		\end{bmatrix}
		=
		\begin{bmatrix}
			\Tsf_{11}\Tsf_{11}^\top & \\
			& \Tsf_{25}\Tsf_{25}^\top
		\end{bmatrix}
		+
		\begin{bmatrix}
			\Tsf_{13}\\
			\Tsf_{23}
		\end{bmatrix}
		\begin{bmatrix}
			\Tsf_{13}^\top & \Tsf_{23}^{\top}
		\end{bmatrix},
		\]
		which is the sum of two semi-positive definite matrices. Therefore \eqref{eqn:parentrelation} implies that for any parent node $\tau$ with children $\alpha$ and $\beta$:
		\[
		\| \wsf^\alpha \|^2 + \| \wsf^\beta\|^2 = \| \wsf^\alpha(J^\alpha_i) \|^2 + \| \wsf^\beta(J^\beta_i )\|^2 \geq \| \wsf^\tau \|^2.
		\]
		Apply the above inequality hierarchically for all parent nodes $\tau$, we have
		\begin{equation}\label{eqn:recursive}
			\| \usf^1\|^2 = \| \wsf^1 \|^2 \leq \| \wsf^2\|^2 + \|\wsf^3\|^2 \leq \cdots \leq \sum_{\tau \text{  is leaf}} \| \wsf^\tau \|^2,
		\end{equation}
		where the summation in the last inequality is over all leaf nodes $\tau$.
		For any leaf $\tau$, the Chebyshev discreization of equation \eqref{eqn:decomp} implies that
		\begin{equation}\label{eqn:localupdate}
			\wsf^\tau(J^\tau_i) = (1 +h \Lsf(J^\tau_i,J^\tau_i))^{-1} \usf^n(J^\tau_i) \,,
		\end{equation}
		where $\Lsf$ is the 2D second order Chebyshev differentiation matrix. 
		It is shown  in \cite{GottliebLustman} that the 1D Chebyshev differentiation matrix $\Esf$ has real, distinct and negative eigenvalues $\sigma_{E,i}$. Notice that $\Lsf(J^\tau_i,J^\tau_i) = \Esf\otimes \Isf + \Isf \otimes \Esf$, we conclude that $\Lsf(J^\tau_i,J^\tau_i)$ is also diagonalizable. In fact, assuming the eigen-decomposition $\Esf = \Vsf_E \Sigma_E \Vsf_E^{-1}$, we must have
		\begin{equation}
			\begin{aligned}
				(\Vsf_\Esf^{-1}\otimes \Vsf_\Esf^{-1}) \Lsf(J^\tau_i,J^\tau_i) (\Vsf_\Esf\otimes \Vsf_\Esf) &=(\Vsf_\Esf^{-1}\otimes \Vsf_\Esf^{-1}) (\Esf\otimes \Isf + \Isf \otimes \Esf) (\Vsf_\Esf\otimes \Vsf_\Esf)  \\
				&= (\Vsf_\Esf^{-1}\otimes \Vsf_\Esf^{-1}) \left((\Vsf_\Esf \Sigma_E \Vsf_\Esf^{-1})\otimes \Isf + \Isf \otimes (\Vsf_\Esf \Sigma_E \Vsf_\Esf^{-1})\right) (\Vsf_\Esf\otimes \Vsf_\Esf) \\
				&=\Sigma_E \otimes \Isf  + \Isf \otimes \Sigma_E.
			\end{aligned}
		\end{equation}
		This implies the eigenvalues of $\Lsf(J^\tau_i,J^\tau_i)$ are pairwise sum $\sigma_{E,i} + \sigma_{E,j} < 0$ and the eigenvectors are pairwise Kronecker product $\Vsf_{\Esf,i}\otimes \Vsf_{\Esf,j}$.
		
		Consequently, there exists the eigen-decomposition of $\Lsf(J^\tau_i,J^\tau_i)$ for any leaf $\tau$,
		\[
		\Lsf(J^\tau_i,J^\tau_i) = \Vsf^\tau_i\Sigma^\tau_i (\Vsf^\tau_i)^{-1} \,,
		\]
		where $\Sigma^\tau_i $ is a diagonal matrix with negative entries and $\Vsf^\tau_i$ contains the eigenvectors. Plug it into equation \eqref{eqn:localupdate}, we have
		\begin{equation}
			\wsf^\tau(J^\tau_i) =\Vsf^\tau_i (1 + h \Sigma^\tau_i )^{-1}(\Vsf^\tau_i)^{-1} \usf^n(J^\tau_i) \,.
		\end{equation}
		Because $\Sigma_i^\tau$ are negative, the L-stability of Euler methods implies that the entries of $(1+h \Sigma^\tau_i )^{-1}$ must have modulus smaller than $1$, therefore
		\[
		\| \wsf^\tau \|^2 = \| \wsf^\tau(J^\tau_i) \|^2 \leq \| \usf^n(J^\tau_i) \|^2 \,.
		\]
		Now combine the above equation with \eqref{eqn:recursive}, we have
		\[
		\| \usf^{n+1}\|^2 = \|\usf^1\|^2 \leq \sum_{\tau \text{  is leaf}}  \| \usf^n(J^\tau_i) \|^2 \leq \|\usf^n\|^2,
		\]
		which implies $\| \mathsf{M}^n \| \leq 1$ and thus the stability of RKHPS.
	\end{proof}
	\begin{remark}
		For general parabolic equations, unfortunately there is no guarantee that spectral approximation to the elliptic operator has real negative eigenvalues. For example, in the strong convection regime or high frequency regime of Helmholtz type operator, the spectral approximation may have imaginary or real positive eigenvalues and in those cases there is no guarantee for stability of RKHPS. Numerical result shows that the RKHPS method is stable for partial differential equations in which the elliptic operator is dominating though.
	\end{remark}
	
	In Figure \ref{fig:varstab} and \ref{fig:varstab2} we have plotted the eigenvalues of the time-stepping map for 1D variable coefficient parabolic equation \eqref{eqn:varelliptic} without or with a HPS hierarchical tree structure. We see that in both cases the RKHPS has eigenvalues bounded by $1$, regardless of the time-step size $\Dt= 1,10^{-2},10^{-4},10^{-6}$. Moreover, for the RKHPS with hierarchical tree, we set the depth of tree $L = 3$ and consequently there are $8$ leaf nodes. By comparing Figure \ref{fig:varstab} and \ref{fig:varstab2}, we see that the tree structure introduces $7$ zero eigenvalues. As demonstrated in the proof stability, such phenomena stems from the continuity flux assumptions in the HPS methods and the number of zero eigenvalues equals to the total number of interfaces. In fact, it can shown that the null space of the time-stepping map consists of functions that are supported on the interfaces. 
	
	\begin{figure}[ht]
		\begin{subfigure}{.5\textwidth}
			\centering
			\includegraphics[width=0.9\linewidth]{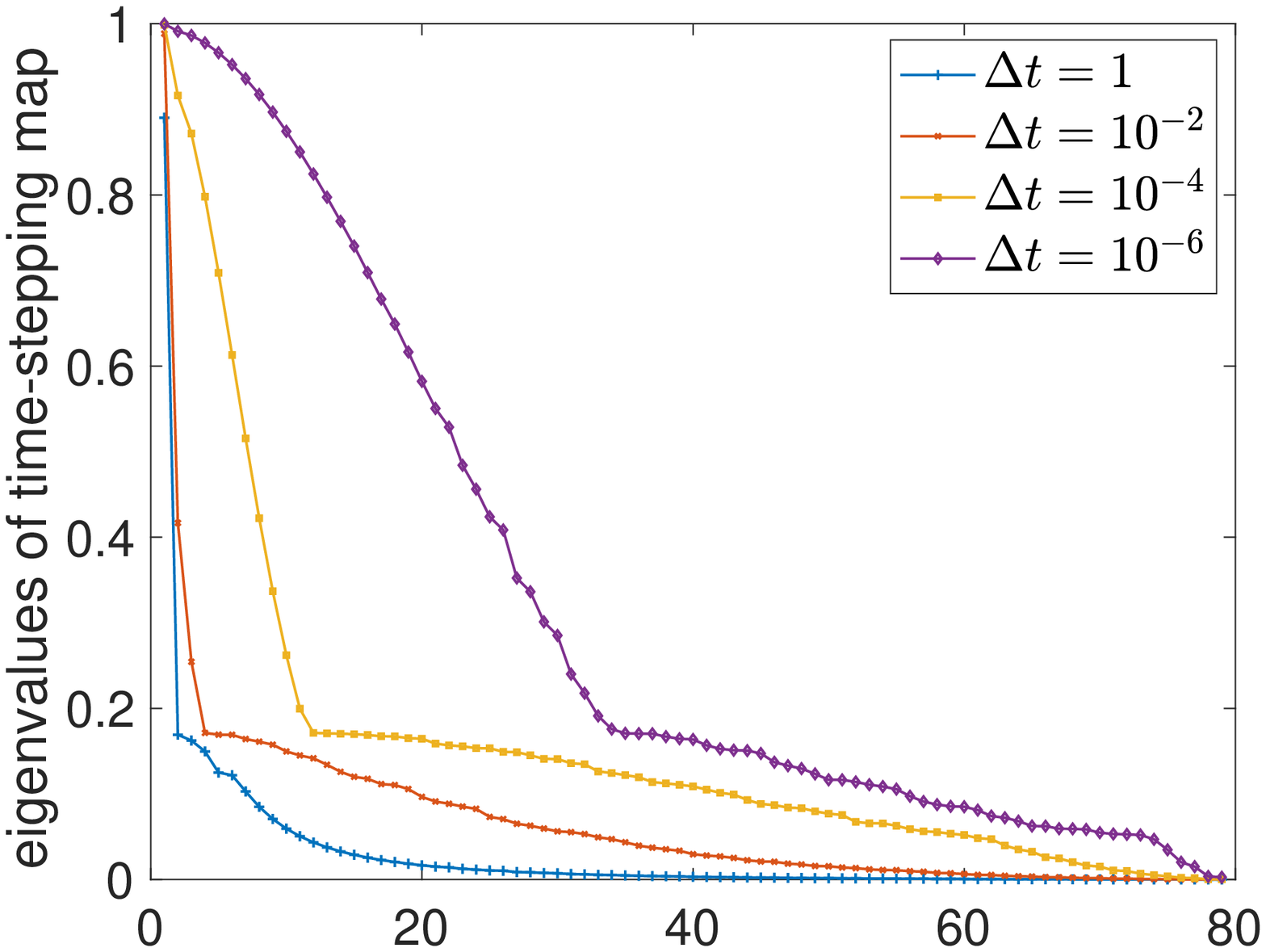}
			\caption{with no hierarchical tree structure}
			\label{fig:varstab}
		\end{subfigure}
		\begin{subfigure}{.5\textwidth}
			\centering
			\includegraphics[width=0.9\linewidth]{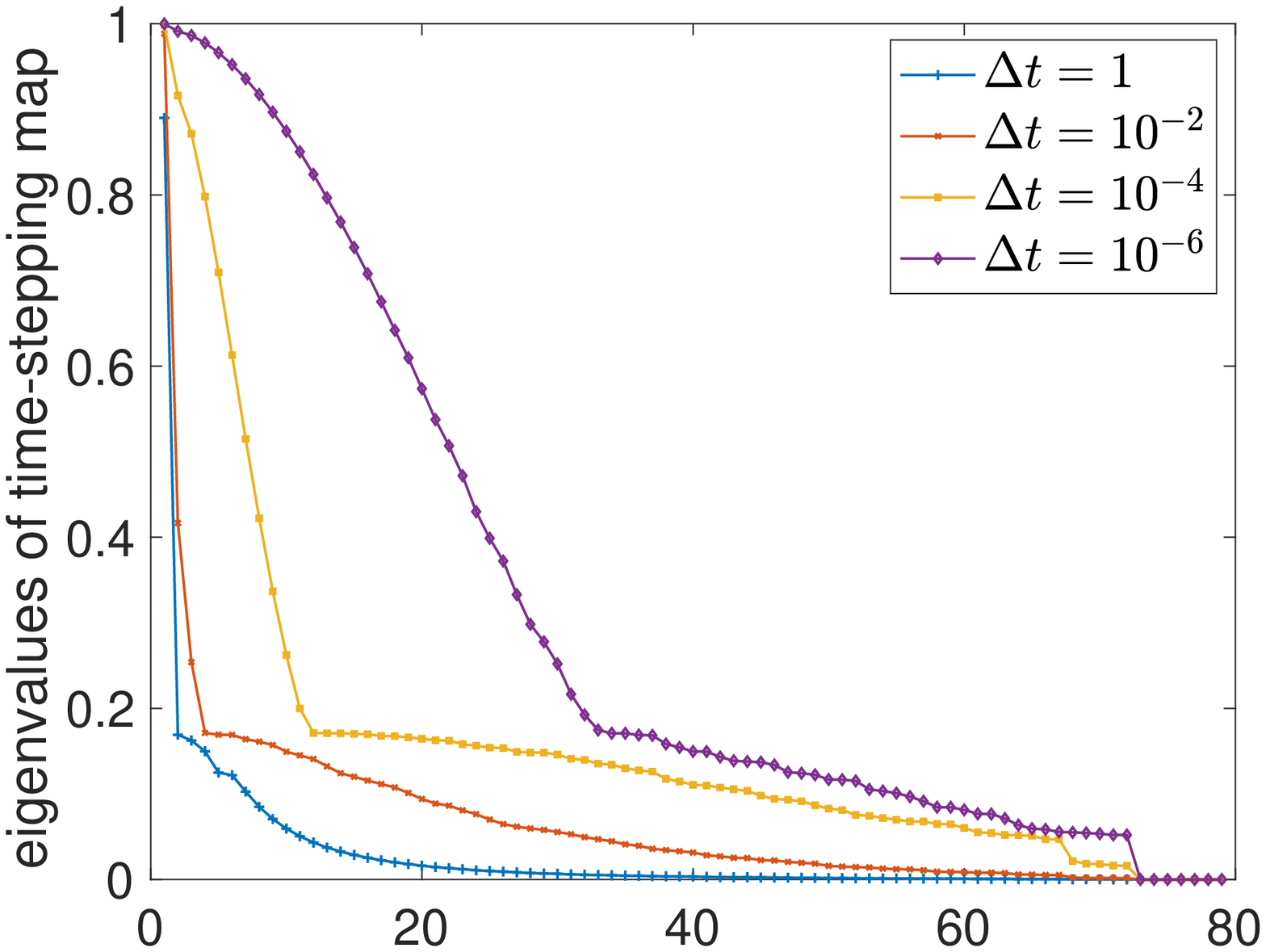}
			\caption{with $8$ leaf nodes in the hierarchical tree}
			\label{fig:varstab2}
		\end{subfigure}
		\caption{Eigenvalues of the time-stepping map for 1D variable coefficient parabolic equation}
	\end{figure}
	
	\section{Numerical tests}\label{sec:numeric}
	
	\subsection{1D convection diffusion equation}
	We then test the convection-diffusion equation with variable convection-coefficient.
	\begin{equation}
		u_t = u_{xx} - k \sin(1 + 1.9\pi x) u_x + q
	\end{equation}
	with initial and boundary conditions
	\begin{equation}\label{eqn:ICBC}
		\begin{aligned}
			u(x,0) &= \sin(1 + 1.7\pi x)\cos(1)\,, &\quad  \text{for all  }x\in[0,2]\\
			u(x,t) &= \sin(1+1.7\pi x)\cos(1 + t^2x)(1+t^3x)\,,&\quad \text{for all  }t\in[0,0.5] \text{  and  } x = 0\text{  or  }2
		\end{aligned}
	\end{equation}
	We first consider the case with no external source $q=0$. We discretize the domain with $32$ leaves and $p=21$ on each leaf node. The approximate solution is computed using a ARK4(3)6L[2]SA-ESDIRK method in \cite{kennedy2003additive}.
	In Figure \ref{fig:diffsoln} and \ref{fig:convsoln} we plot time stamps of approximate solution with $k=1$ and $k=100$. In the strong diffusion regime $k=1$, the solution profiles \ref{fig:diffsoln} are similar to that of the heat equation. In comparison, in the strong convection regime $k=100$, the solution in Figure \ref{fig:convsoln} quickly forms shocks at point $x=0.3588$ and $x=1.4114$. These two points are exactly where the convection coefficient $\sin(1+1.9\pi x)$ changes from positive to negative. 
	\begin{figure}[ht]
		\begin{subfigure}{.5\textwidth}
			\centering
			\includegraphics[width=0.9\linewidth]{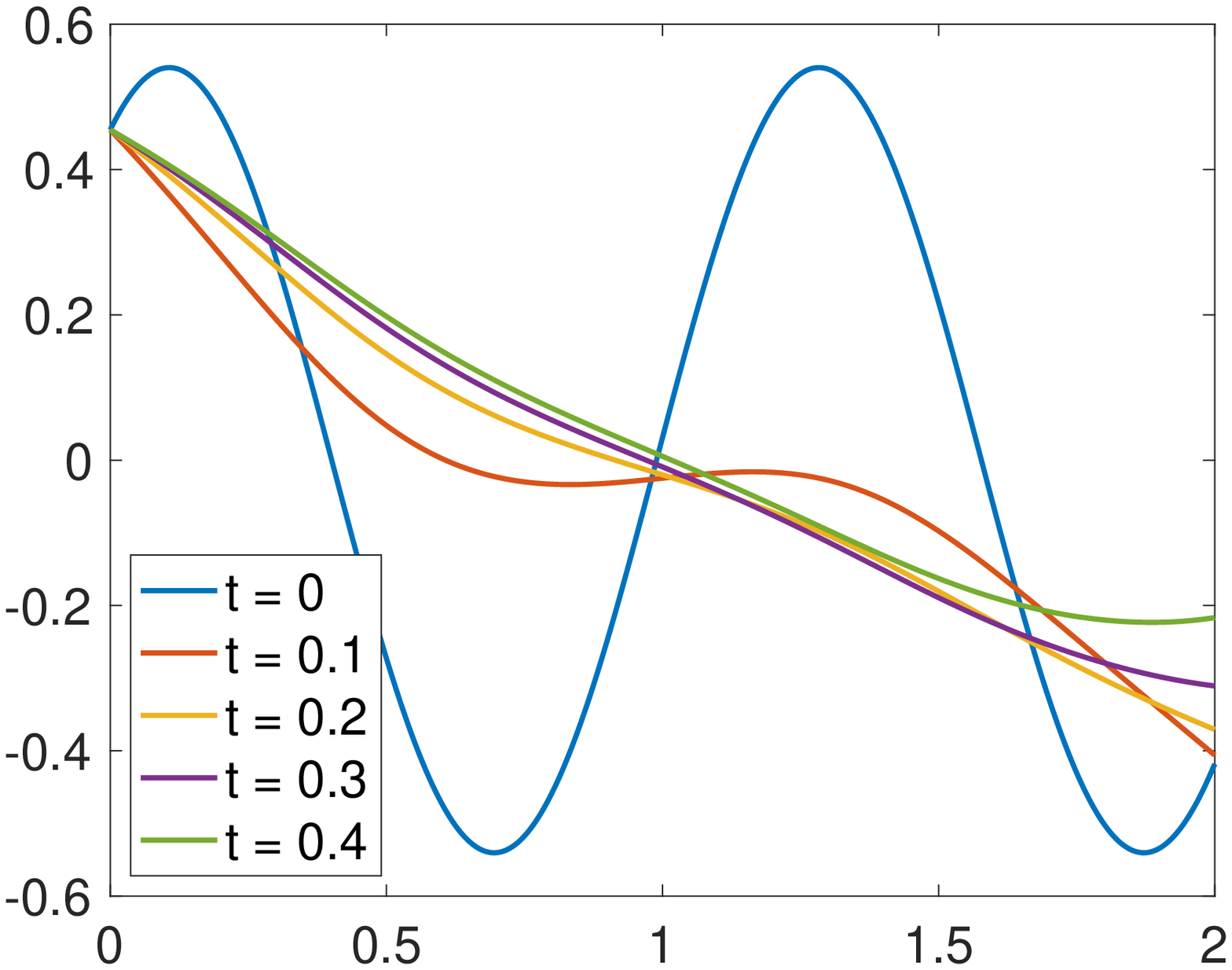}
			\caption{strong diffusion regime}
			\label{fig:diffsoln}
		\end{subfigure}
		\begin{subfigure}{.5\textwidth}
			\centering
			\includegraphics[width=0.9\linewidth]{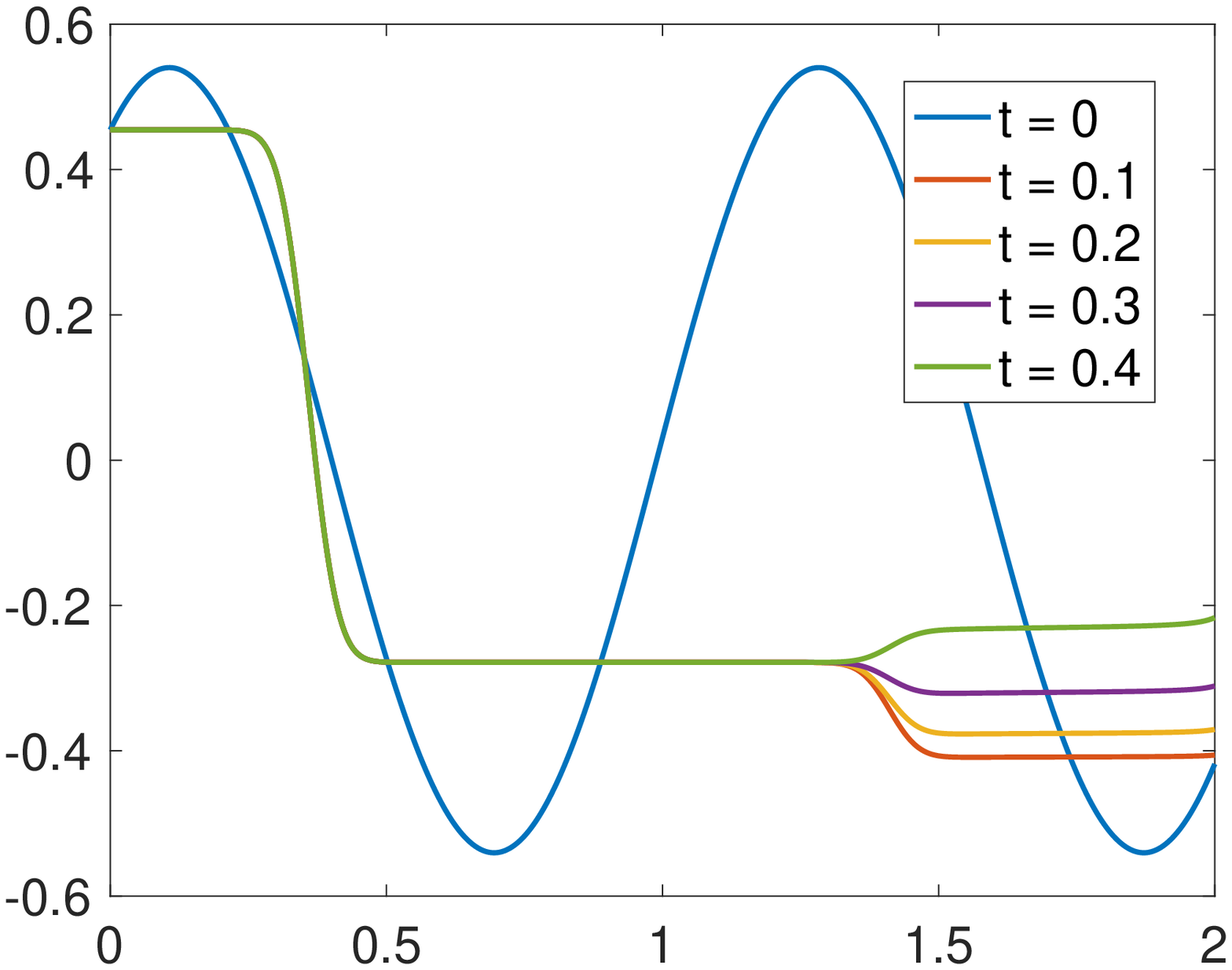}
			\caption{strong convection regime}
			\label{fig:convsoln}
		\end{subfigure}
		\caption{time stamps of convection diffusion equation}
	\end{figure}
	We first set $k=1$ and test the diffusion dominated case. We plot the case of inhomogeneous and homogeneous BC in Figure \ref{fig:cvdfconv} and \ref{fig:cvdfconvhomo}. The plots are again similar to that of the heat equation upto some minor difference. Then we set $k=100$ and plot in Figure \ref{fig:cvdfconv100} and \ref{fig:cvdfconvhomo100}. In this case, the order of convergence drops to $3$rd order regardless of different formulations or type of boundary conditions.
	
	\begin{figure}[ht]
		\begin{subfigure}{.5\textwidth}
			\centering
			\includegraphics[width=0.9\linewidth]{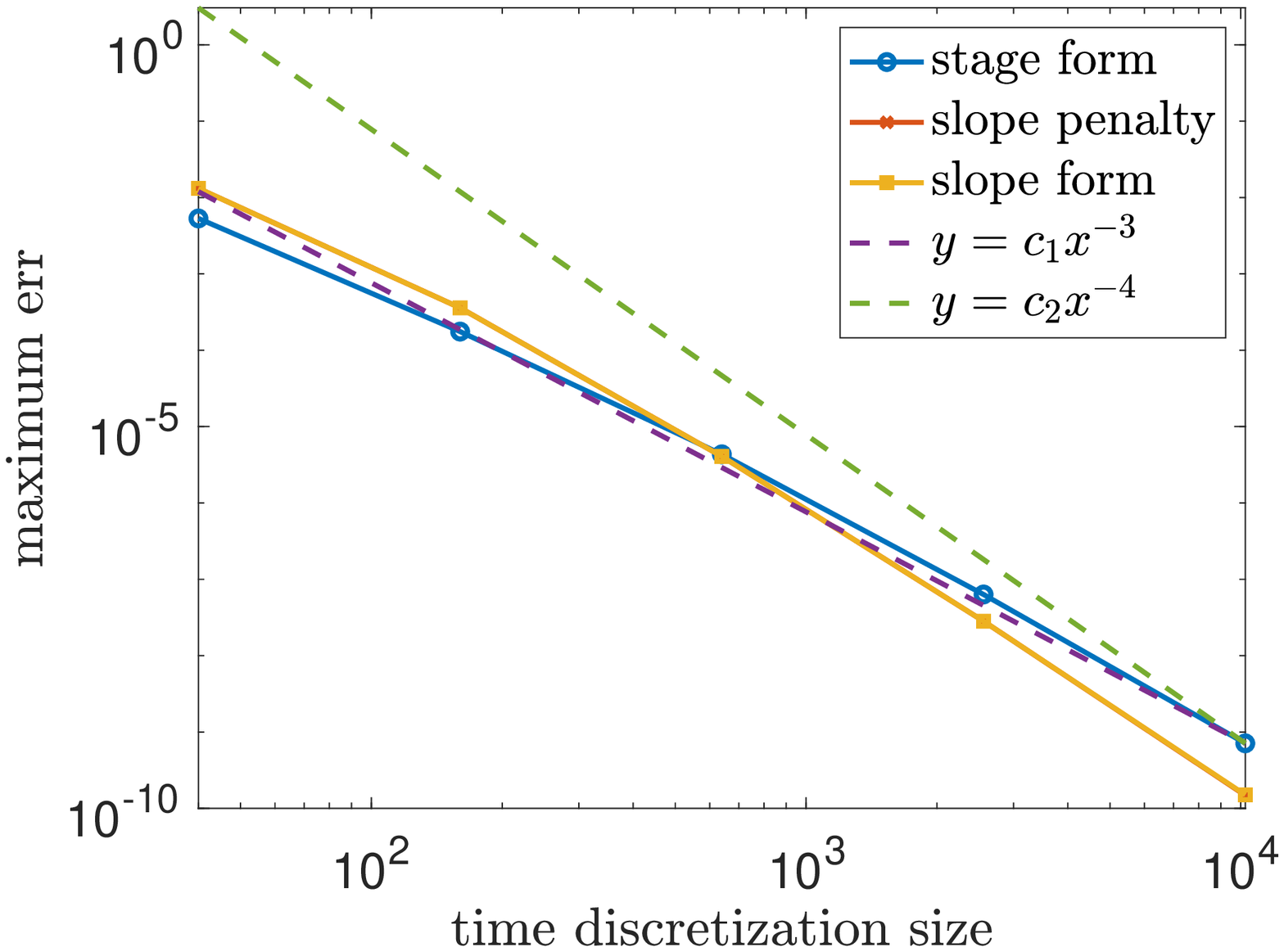}
			\caption{inhomogeneous BC}
			\label{fig:cvdfconv}
		\end{subfigure}
		\begin{subfigure}{.5\textwidth}
			\centering
			\includegraphics[width=0.9\linewidth]{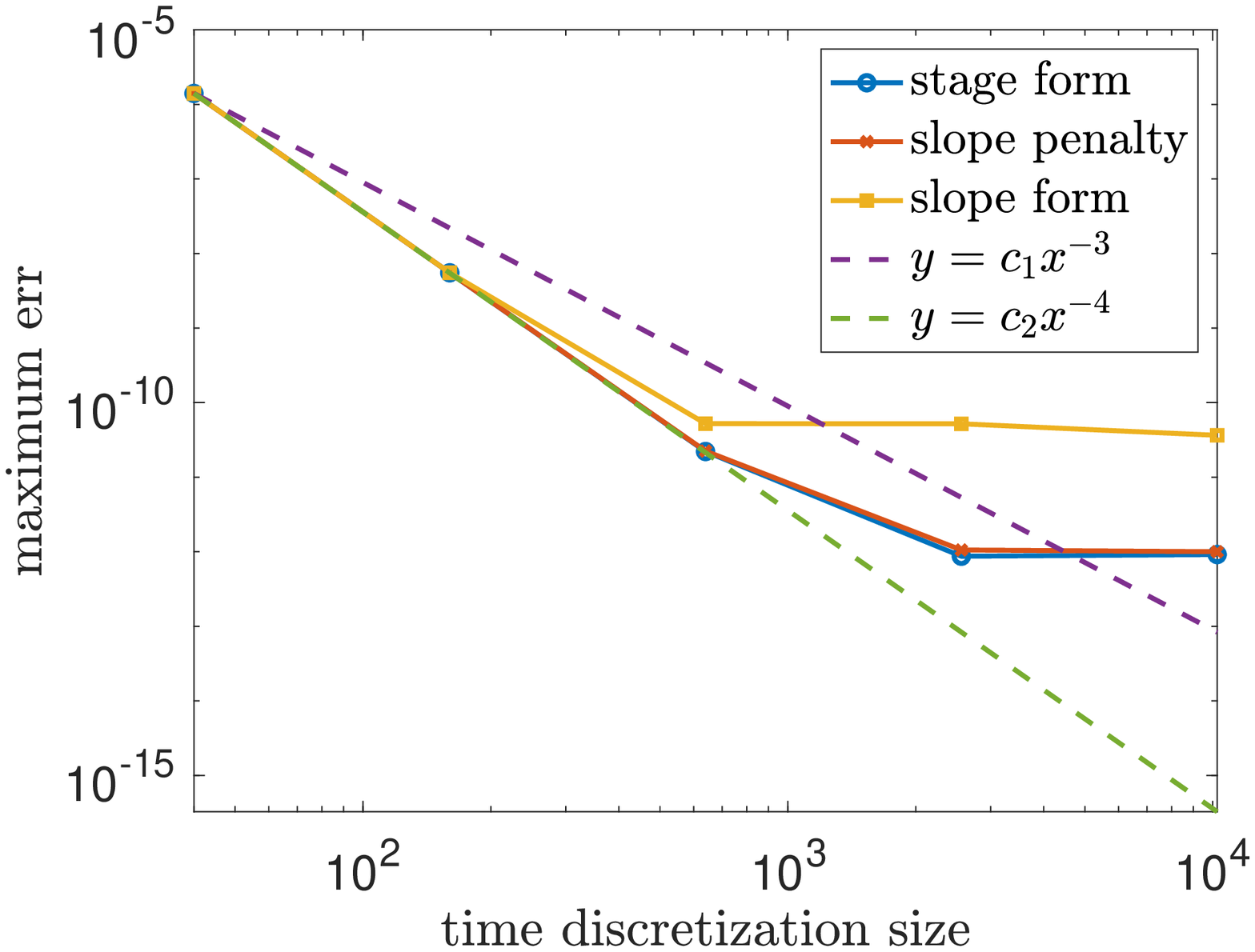}
			\caption{homogeneous BC}
			\label{fig:cvdfconvhomo}
		\end{subfigure}
		\caption{convergence test of convection diffusion equation with $k=1$}
	\end{figure}
	
	\begin{figure}[ht]
		\begin{subfigure}{.5\textwidth}
			\centering
			\includegraphics[width=0.9\linewidth]{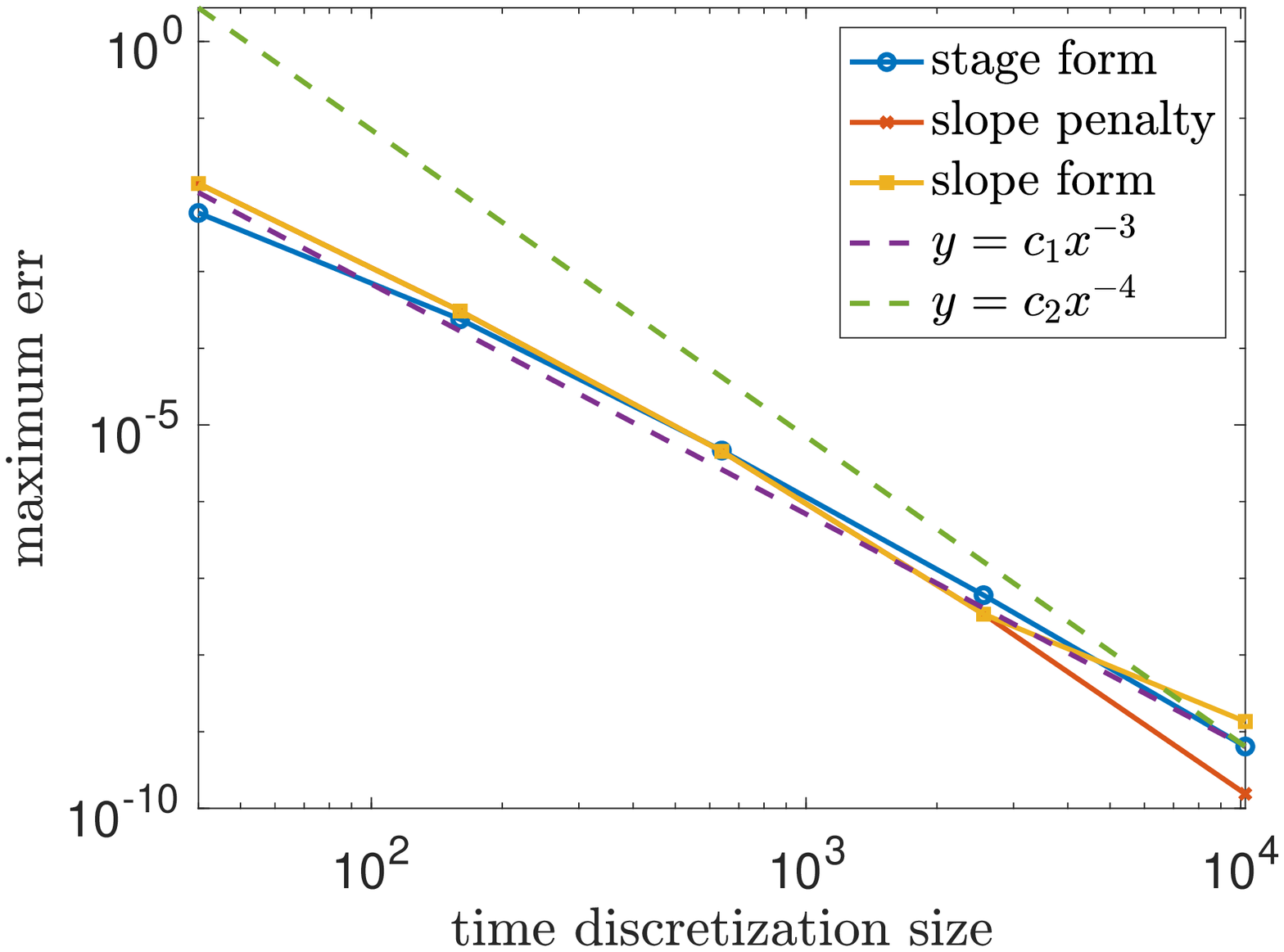}
			\caption{inhomogeneous BC}
			\label{fig:cvdfconv100}
		\end{subfigure}
		\begin{subfigure}{.5\textwidth}
			\centering
			\includegraphics[width=0.9\linewidth]{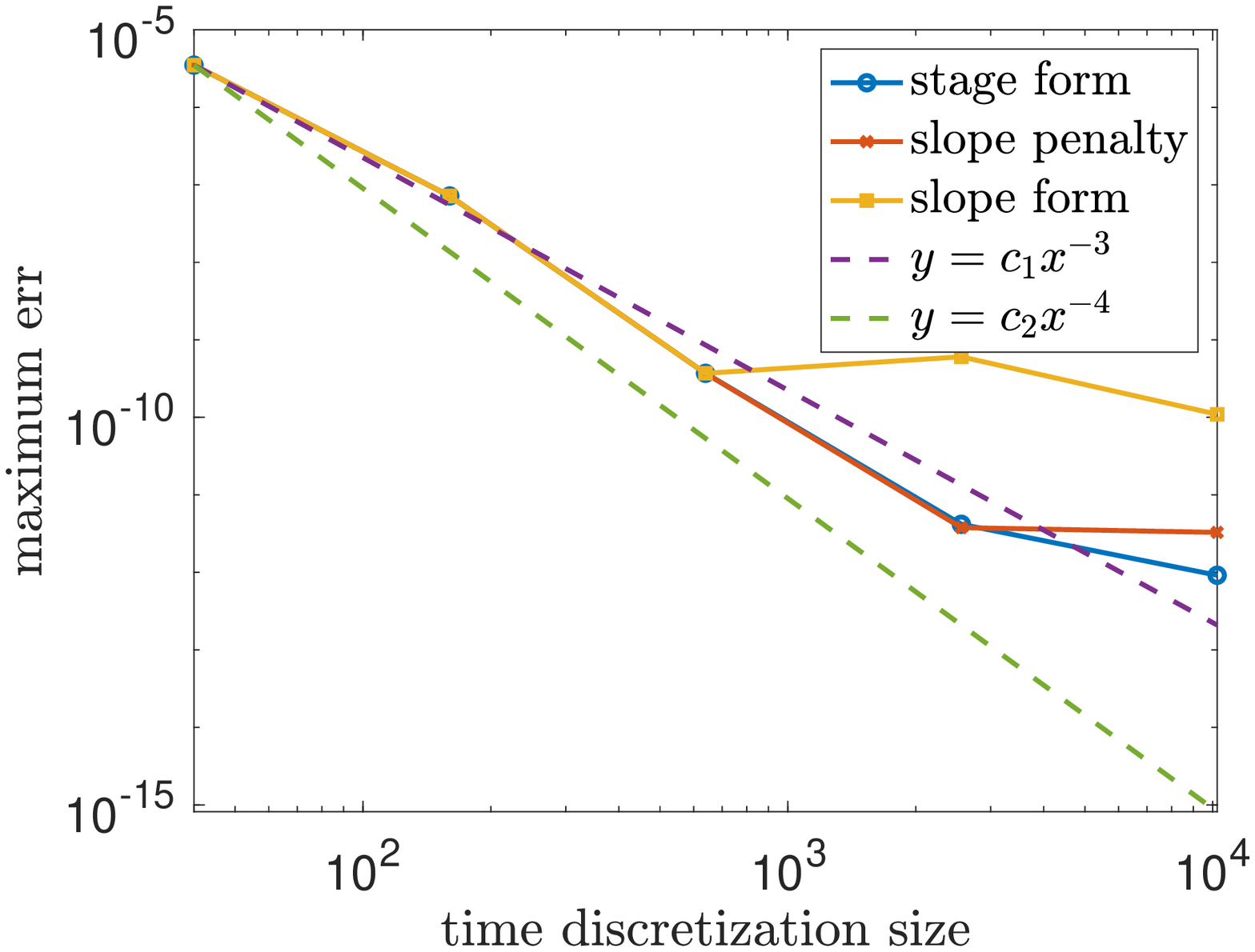}
			\caption{homogeneous BC}
			\label{fig:cvdfconvhomo100}
		\end{subfigure}
		\caption{convergence test of convection diffusion equation with $k=100$}
	\end{figure}

	\subsection{1D variable coefficient parabolic equation}
	We test variable coefficient parabolic equation in this section with same discretization and ESDIRK method as the previous subsection.
	\begin{equation}\label{eqn:varelliptic}
		u_t = \partial_x(a\partial_x u) + \kappa^2 u + q
	\end{equation}
	where $a(x) = 1 + 0.9\sin(1 + 1.9\pi x)$ is the inhomogeneous medium conductivity. We set $\kappa=1$ so the equation is diffusion dominated. 	We first calculate a typical solution with no external source $q=0$ and with initial and boundary conditions set as in \eqref{eqn:ICBC}.
	The time stamps of the approximate solution is plotted in Figure \ref{fig:4soln}. One can see that the solution quickly changes from the sine profile to a monotone temperature diffusion profile. Also notice that the solution is nearly a constant on regions where the medium conductivity achieves large values.
	
	\begin{figure}[ht]
		\centering
		\includegraphics[width=0.6\linewidth]{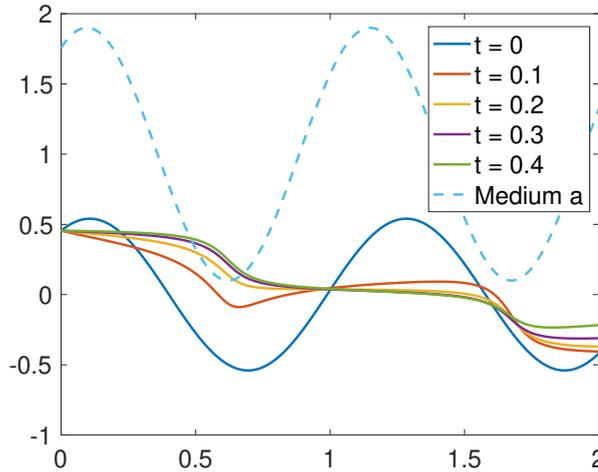}
		\caption{time stamps of approximate solution from $t=0$ to $t=0.5$.}
		\label{fig:4soln}
	\end{figure}
	In Figure \ref{fig:varconv} and \ref{fig:varconvhomo} we plot the convergence rate for inhomogeneous BC and homogeneous BC cases respectively. The plots more or less resembles that of the heat equation. For inhomogeneous case, we obtain asymptotically $3$rd order convergence for stage formulation and nearly $4$th order for slope formulations. For homogeneous case, all three methods obtains $4$th order convergence before it gets saturated at magnitude of $10^{-10}$ to $10^{-12}$.
	
	\begin{figure}[ht]
		\begin{subfigure}{.5\textwidth}
			\centering
			\includegraphics[width=0.9\linewidth]{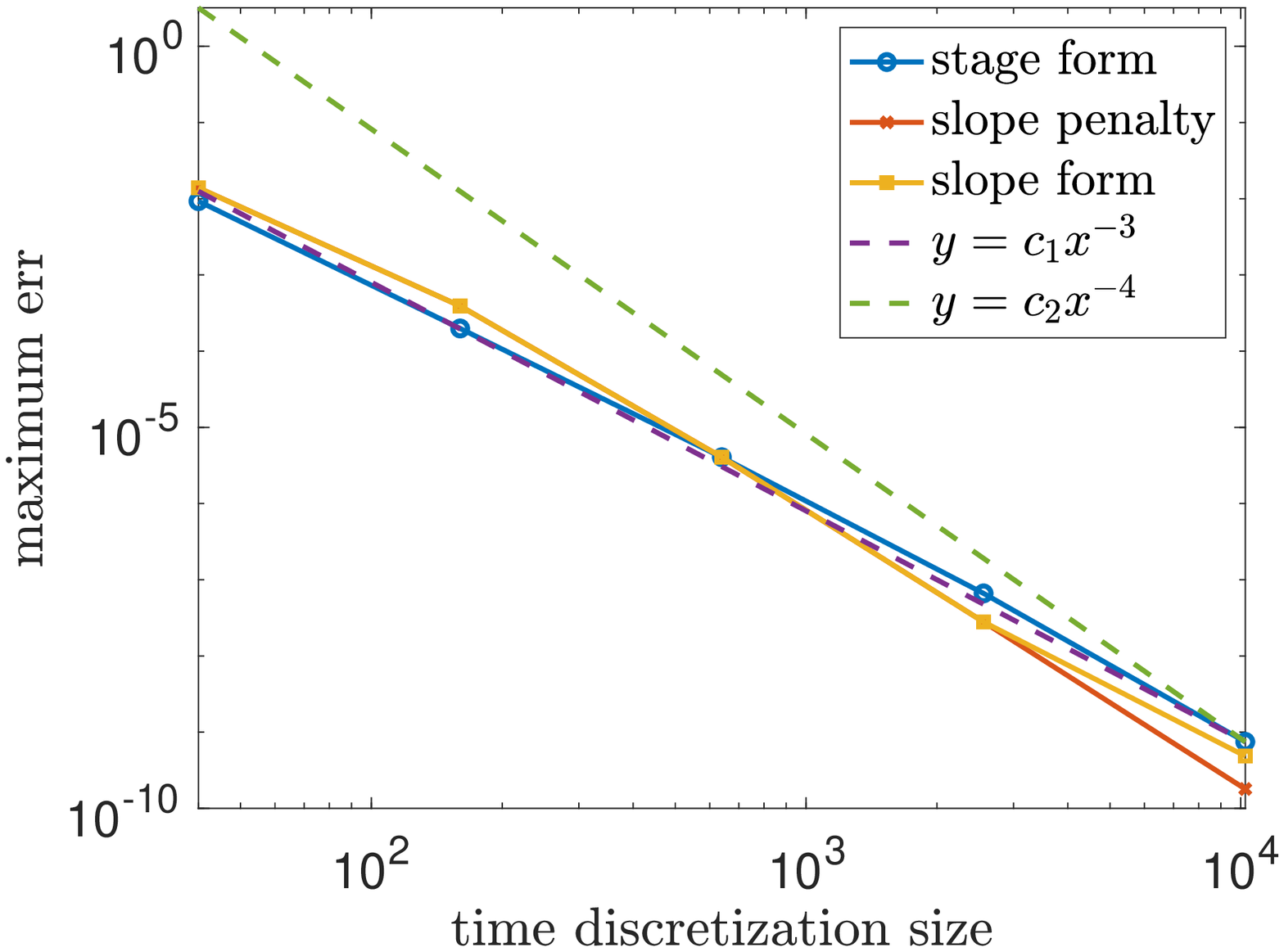}
			\caption{inhomogeneous BC}
			\label{fig:varconv}
		\end{subfigure}
		\begin{subfigure}{.5\textwidth}
			\centering
			\includegraphics[width=0.9\linewidth]{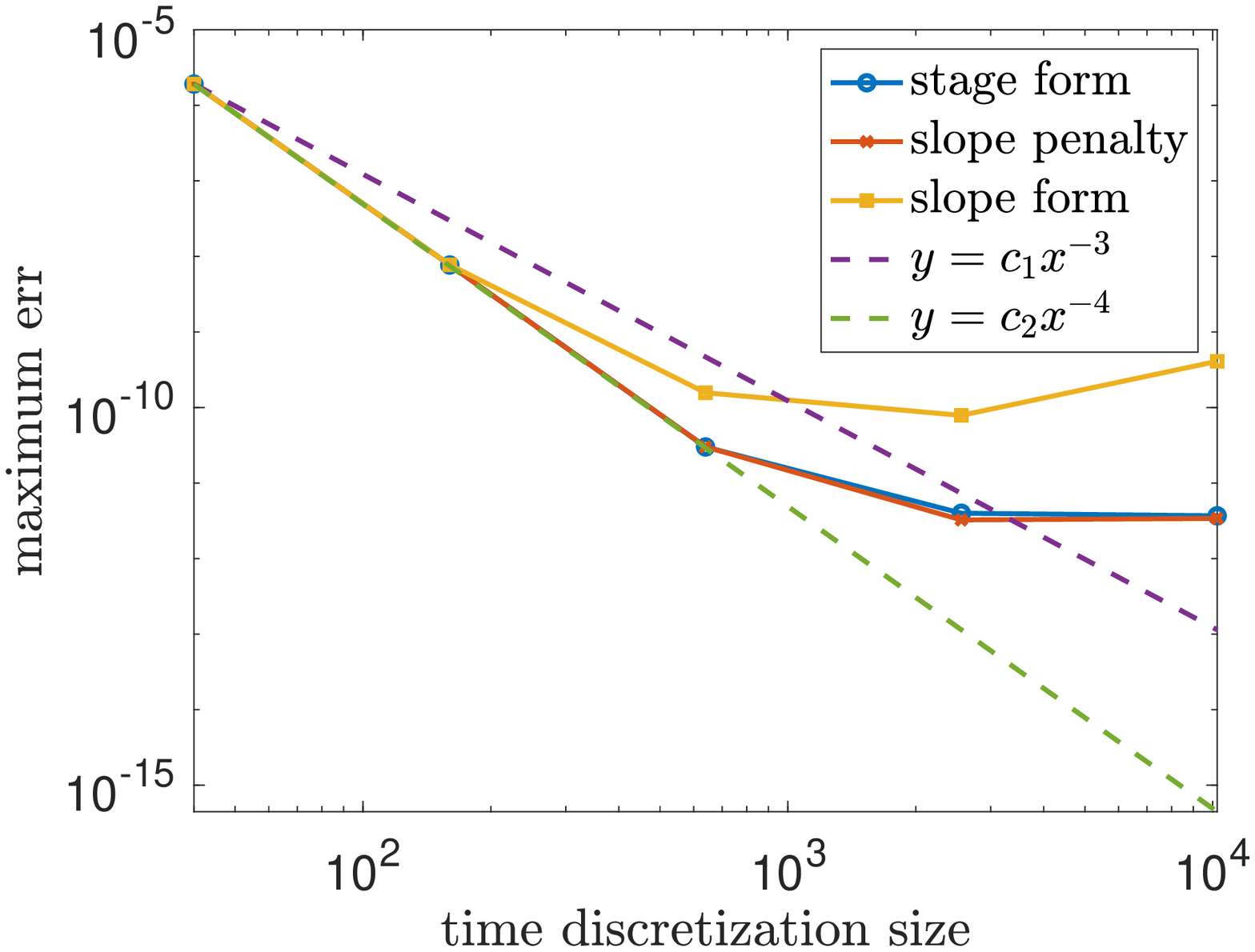}
			\caption{homogeneous BC}
			\label{fig:varconvhomo}
		\end{subfigure}
		\caption{convergence test of variable coefficient parabolic equation}
	\end{figure}

	\subsection{2D heat equation}
	We consider the 2D heat equation in this section
	\begin{equation}\label{eqn:heat2}
		u_t = u_{xx} + u_{yy} + q \,,
	\end{equation}
	where a suitable source $q$, initial and boundary conditions that are compatible with the exact solution
	\[
	u(t,x,y) = \sin(\pi x)\exp(-t(y-\frac{1}{2})^2) \,,
	\]
 in the inhomogeneous boundary condition case, or
 \[
    u(t,x,y) = \sin(2 \pi x) \sin( 2 \pi y) \exp(-t(x+y)) \,,
 \]
 in the homogeneous boudnary condition case.
 
	The domain $\Omega = [0,1]^2$ is divided into $8\times 8$ nodes with $p=21$. The approximate solution is computed using the ARK4(3)6L[2]SA-ESDIRK method in \cite{kennedy2003additive}.
	For both the stage and slope formulation, we plot the maximum error with different time discretizations in Figure \ref{fig:2Dconv1} and \ref{fig:2Dconv2}. In the case with inhomogeneous BC, the (penalized) slope formulation has minor order reduction while the stage formulation lost about one order of accuracy. In comparison, in the homogeneous BC cases, both formulations have same order of accuracy. 
	
	\begin{figure}[ht]
		\begin{subfigure}{.5\textwidth}
			\centering
			\includegraphics[width=0.9\linewidth]{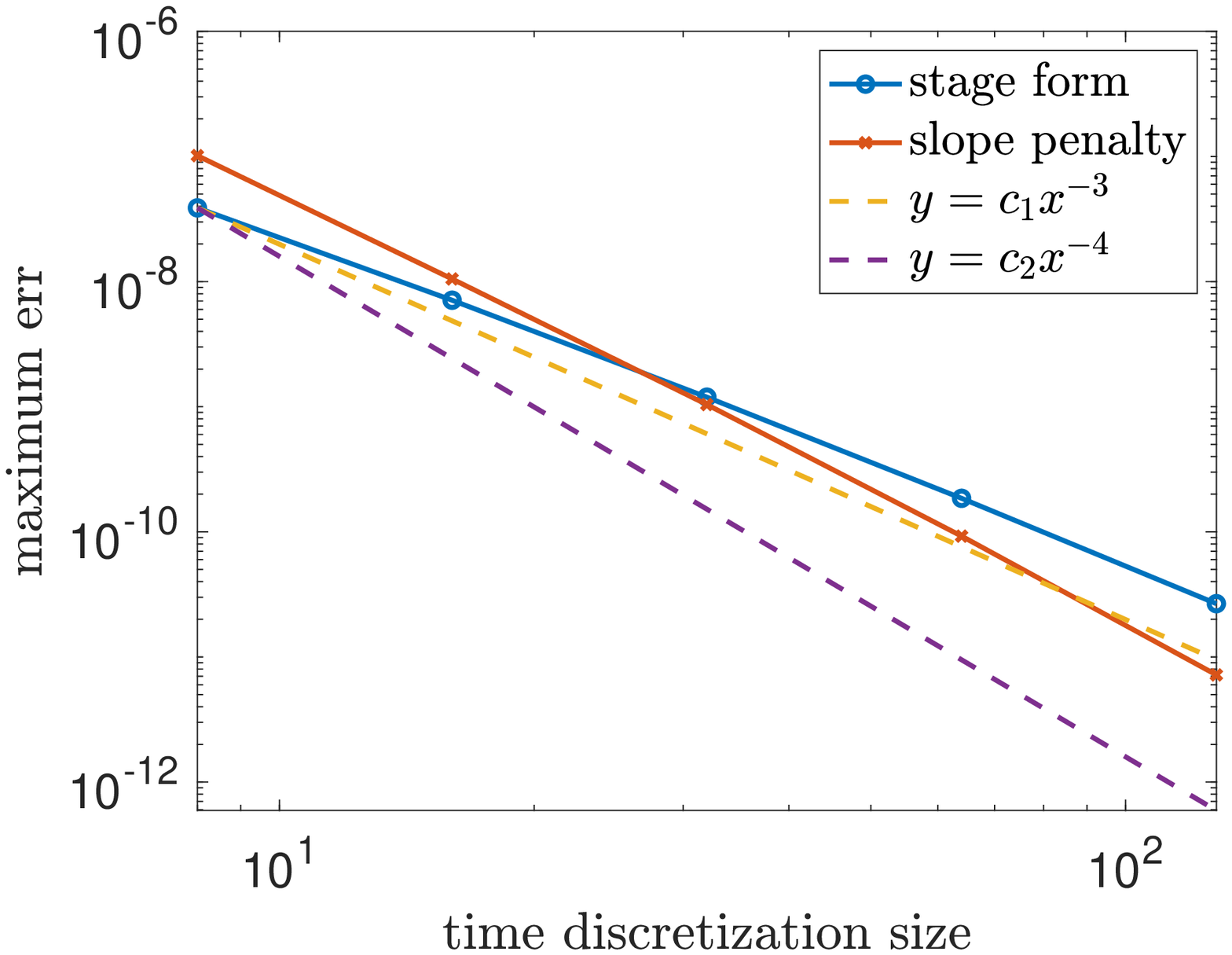}
			\caption{inhomogeneous BC}
			\label{fig:2Dconv1}
		\end{subfigure}
		\begin{subfigure}{.5\textwidth}
			\centering
			\includegraphics[width=0.9\linewidth]{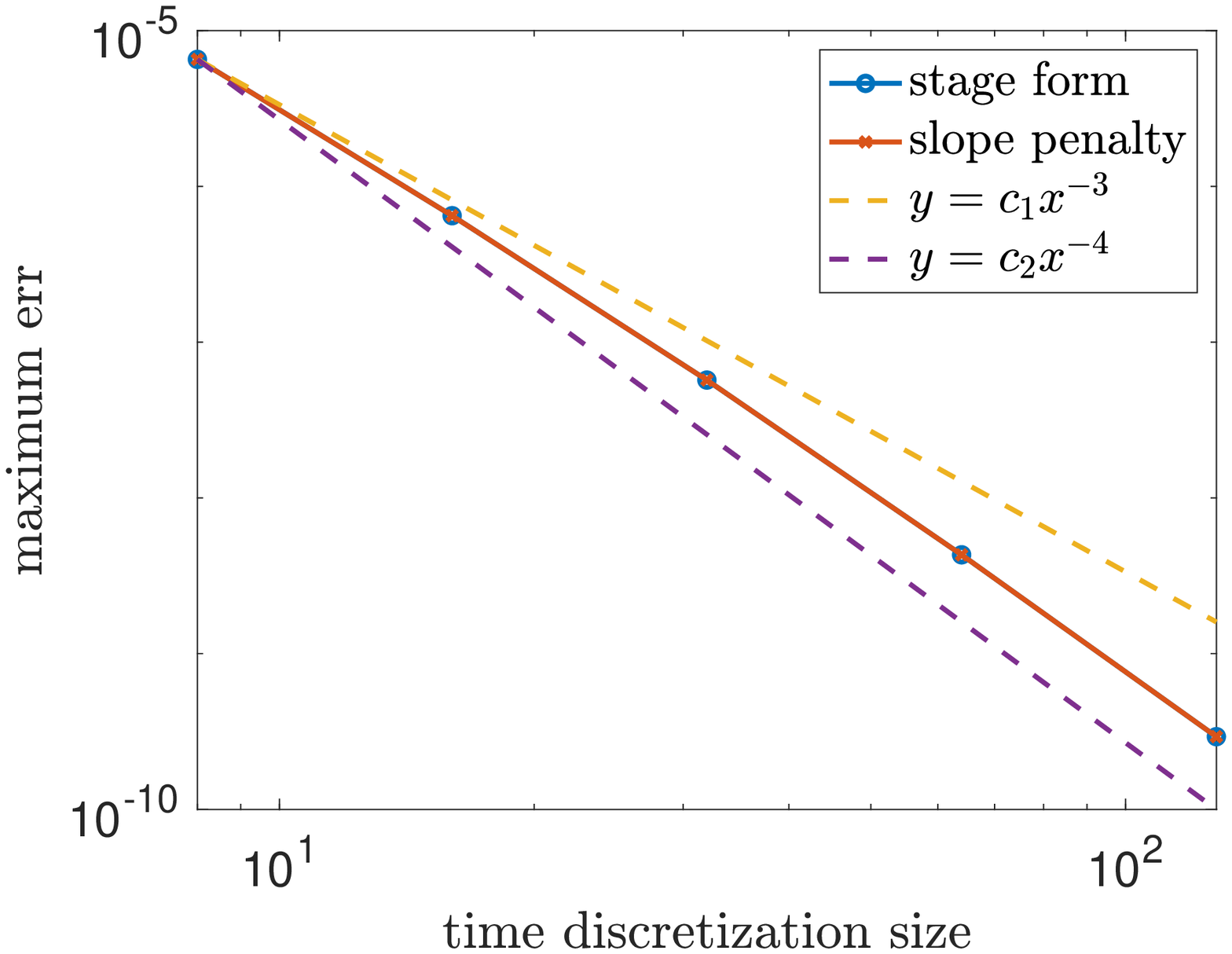}
			\caption{homogeneous BC}
			\label{fig:2Dconv2}
		\end{subfigure}
		\caption{convergence test of heat equation}
	\end{figure}

\subsubsection*{$O(N\log(N))$ time complexity}\label{sec:compareFEM}
We plotted the CPU time of each time step for solving the heat equation in Figure \ref{fig:FEMcompare}. In particular, the computational domain is discretized with $4\times 4$, $8 \times 8$,\ldots, $64 \times 64$ subdomains, each with $11\times 11$ Chebyshev nodes. The total number of degrees of freedom ranges from $1681$ to $410,881$. We then compare the proposed method with the finite element method (FEM). For the FEM method, the space is discretized with the same number of degree of freedom with no hierarchical structure. For simplicity, we used the backward-Euler method for the time discretization. The FEM solution is obtained using LAPACK DGB-FA and DGB-SL with a PLU decomposition to take advantage of the banded structure of the stiffness matrix. We note that there are more advanced ways to solve the linear system associated with high order FEM discretizations. A recent paper discussing such techniques is \cite{lor}.
In Figure \ref{fig:FEMcompare}, it is shown that the solve stage time of the proposed method scales as $O(N\log(N))$ while that of FEM is slightly better than $O(N^2)$. 

        \begin{figure}[ht]
		\centering
		\includegraphics[width=0.6\linewidth]{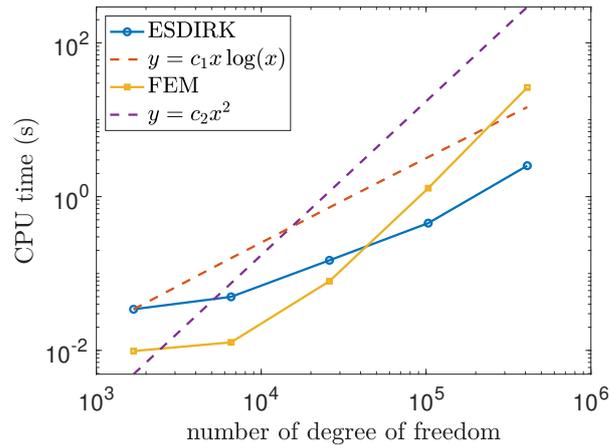}
		\caption{CPU time comparison between our method and FEM for solving one time step for the heat equation}
		\label{fig:FEMcompare}
	\end{figure}
 
	\subsection{2D Burgers equation}
	In this section, we consider the following 2D Burgers equations:
	\begin{equation}
		\begin{aligned}
			u_t &= \varepsilon (u_{xx}+u_{yy}) - (uu_x+vu_y) \,,\\
			v_t &= \varepsilon (u_{xx}+u_{yy}) - (uv_x+vv_y) \,.
		\end{aligned}
	\end{equation}
	We use the same discretization in this example and ARK4(3)6L[2]SA-ESDIRK for the Analytic solution test. A higher order method ARK5(4)8L[2]SA-ESDIRK (\cite{kennedy2003additive}) is used for the rotating flow test to compute the approximate soluitons.
	We investigate the convergence rate for different typical solution to the Burgers equation and show that RKHPS has high order convergence and stability, though loss of order convergence occurs for inhomogeneous boundary condition problem. A rotating flow example is also provided to demonstrate that RKHPS can capture a sharp transition region in the fluid.

	\subsubsection{Analytic solution test}
	The Burgers equations are solved from $t=0$ and $T=2$ on a unit square $[0,1]^2$. We partition the domain $[0,1]^2$ into $8\times 8=64$ smaller squares and each of them is further discretized with order $p=21$ Chebyshev nodes. In total there are $25921$ space grid points. The maximum error $\| u(T,\cdot)-u_\text{exact}(T,\cdot)\|_\infty$ are calculated on all grid points except those corner points.
	The viscosity is set as $\varepsilon= 0.1$. In this test, the slope formulation is not applicable and only stage formulation is used here. We study the convergence rate for different type of exact solutions at the terminal time $T=2$.
	\begin{description}
		\item[Traveling wave solution] The exact solution is of the following form:
		\begin{equation}
			\begin{aligned}
				u_\text{travel}(x,y,t) &= \frac{3}{4} - \frac{1}{4\left(1+\frac{\exp(4y-4x-t)}{32\varepsilon}\right)} \,,\\
				v_\text{travel}(x,y,t) &= \frac{3}{4} + \frac{1}{4\left(1+\frac{\exp(4y-4x-t)}{32\varepsilon}\right)} \,.\\
			\end{aligned}
		\end{equation}
		In Figure \ref{fig:2DBurgers1}, we plot the maximum error for both $u$ and $v$, the convergence rate is near $3$rd order due to the order reduction in inhomogeneous BC solutions.
		\item[Highly diffusive solution] In this case, we consider the following exact solution:
		\begin{equation}
			\begin{aligned}
				u_\text{diffusive}(x,y,t)&=-\frac{4\pi\varepsilon \exp(-5\pi^2\varepsilon t)\cos(2\pi x)\sin(\pi y)}{2+\exp(-5\pi^2 \varepsilon t)\sin(2\pi x)\sin(\pi y) } \,,\\
				v_\text{diffusive}(x,y,t)&=-\frac{2\pi\varepsilon \exp(-5\pi^2\varepsilon t)\sin(2\pi x)\cos(\pi y)}{2+\exp(-5\pi^2 \varepsilon t)\sin(2\pi x)\sin(\pi y) } \,.
			\end{aligned}
		\end{equation}
		In Figure \ref{fig:2DBurgers2}, we plot the maximum error for $u$ and $v$. The convergence rate of both are close to $4$th order and gradually drop to $3$rd order as the time discretization get finer.
	\end{description}
	
	\begin{figure}[ht]
		\begin{subfigure}{.5\textwidth}
			\centering
			\includegraphics[width=0.9\linewidth]{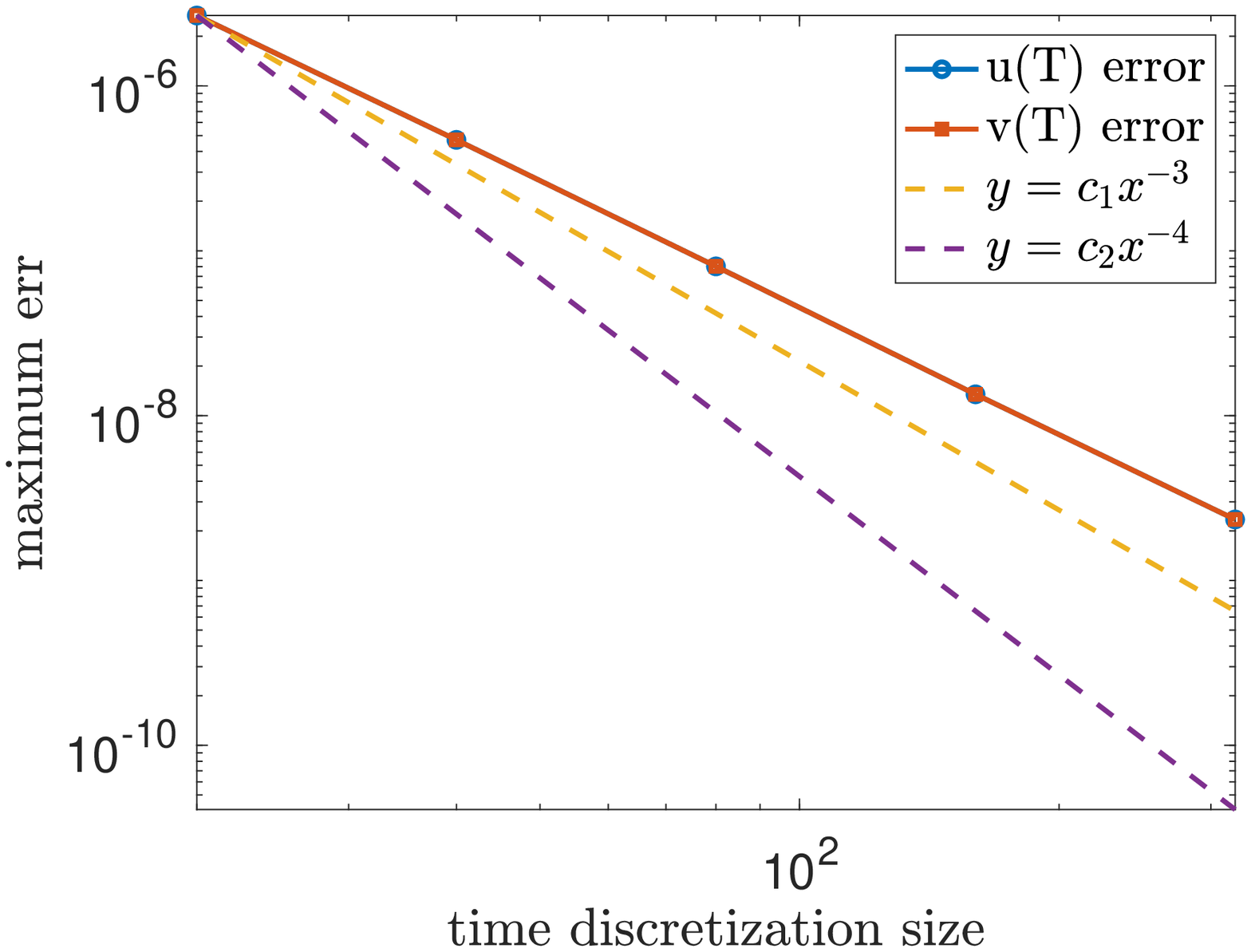}
			\caption{traveling wave solution}
			\label{fig:2DBurgers1}
		\end{subfigure}
		\begin{subfigure}{.5\textwidth}
			\centering
			\includegraphics[width=0.9\linewidth]{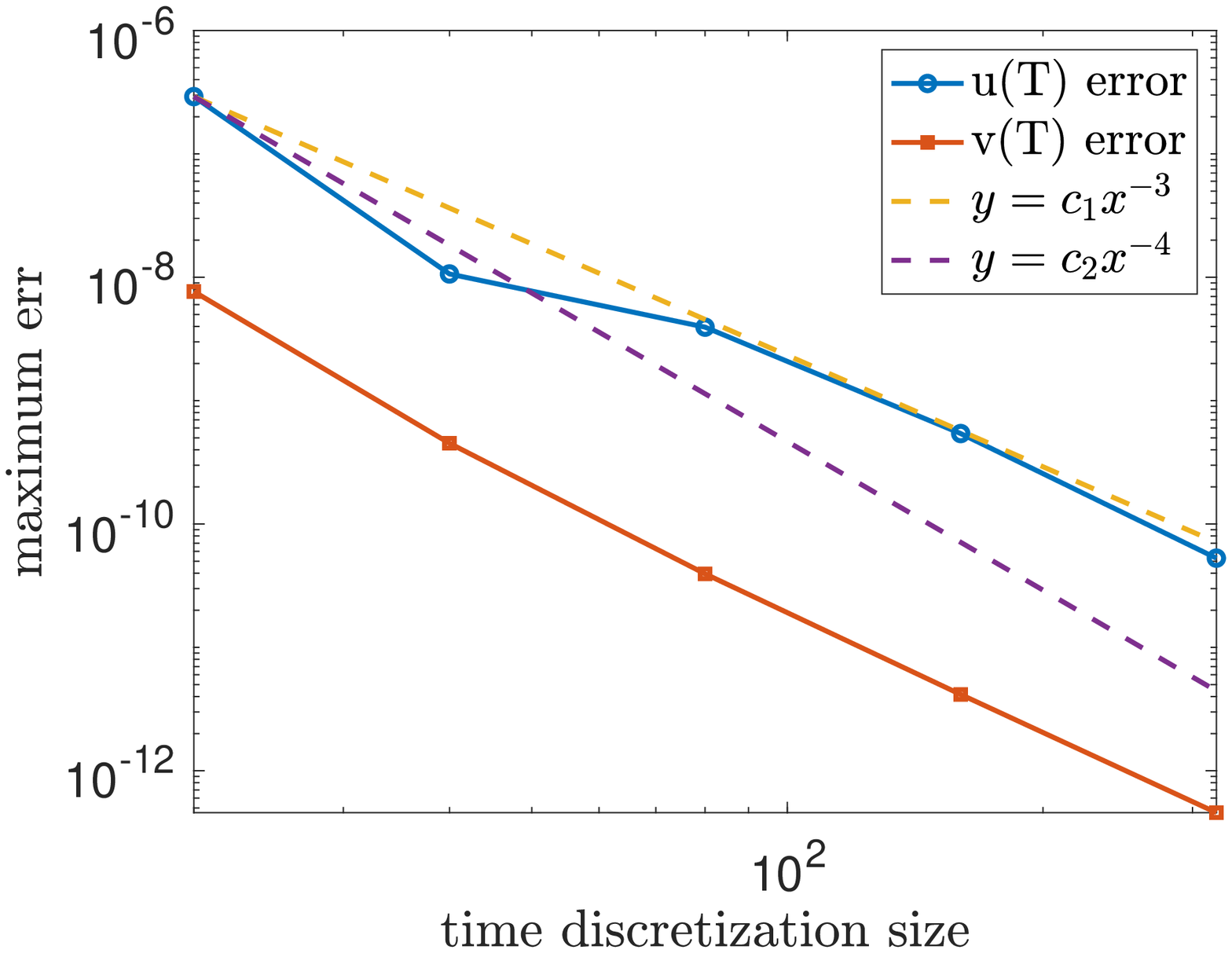}
			\caption{highly diffusive solution}
			\label{fig:2DBurgers2}
		\end{subfigure}
		\caption{convergence test of Burgers equation}
	\end{figure}	
	
	\subsubsection{Rotating flow test}
	In this test, we consider a rotating flow problem on the domain $[-\pi,\pi]^2$ with terminal time $T = 1.5$. The viscosity is set to be low $\varepsilon = 0.005$ and we set initial condition as
	\[
	\begin{aligned}
		u &=-5y\exp(-3(x^2 + y^2))\,,\\
		v &=5x\exp(-3(x^2+y^2)) \,.
	\end{aligned}
	\]
	and set no-slip boundary conditions. The domain is discretized with $24\times 24$ leaf nodes with $p=24$ and the time is discretized with $\tau = 0.01$. To capture the shock like transition region, we used a $5$th order ESDIRK method.
	In Figure \ref{fig:rotate}, we plot the contour of velocity $[u,v]^\top$ at time $t = 0.01,0.51$ and $1.01$. We can see that two semicircle fluid are rotating and gradually expanding to a larger circle, eventually they form a shock near the circle. 
	
	\begin{figure}[ht]
		\begin{subfigure}{.32\textwidth}
			\centering
			\includegraphics[width=0.9\linewidth]{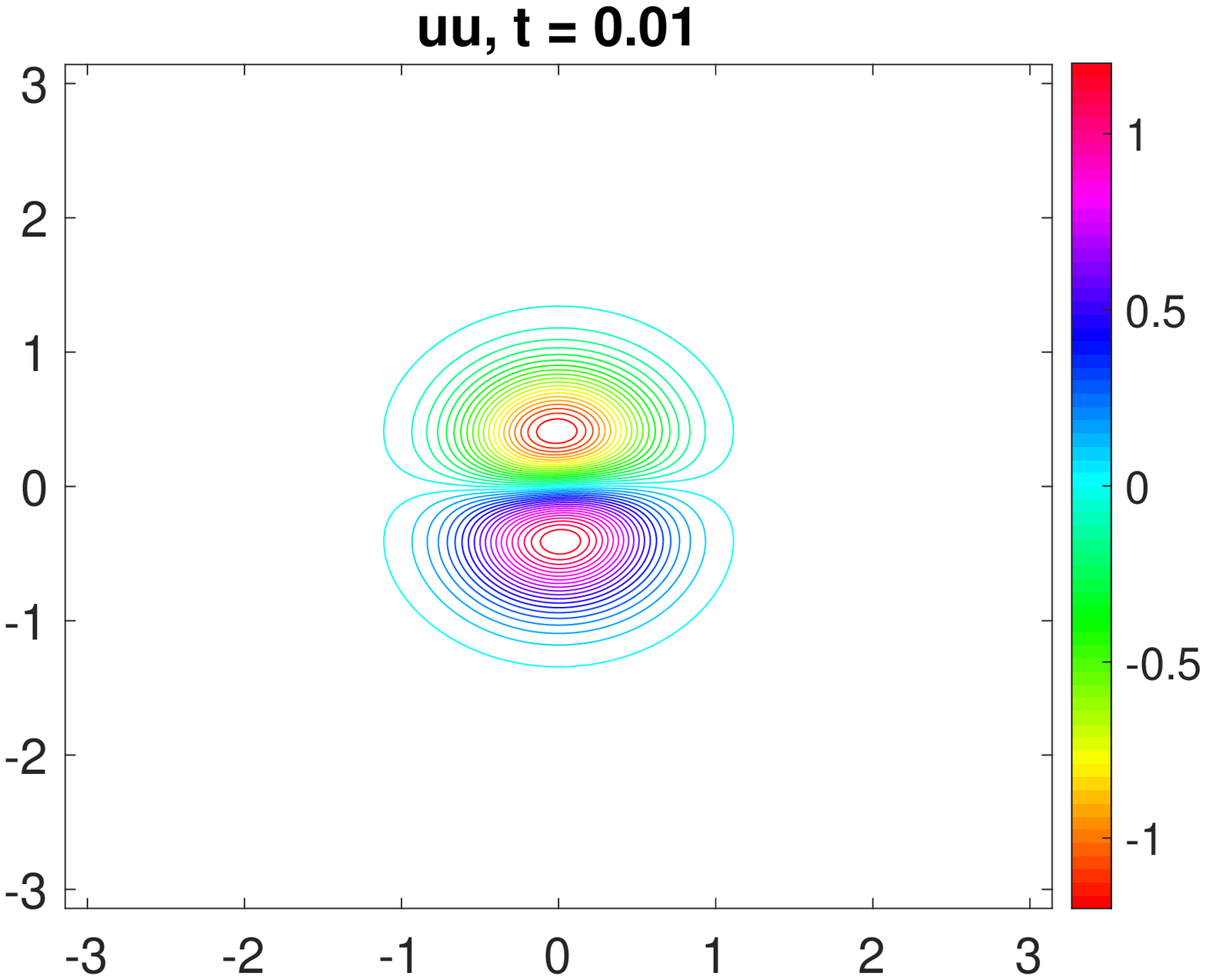}
		\end{subfigure}
		\begin{subfigure}{.32\textwidth}
			\centering
			\includegraphics[width=0.9\linewidth]{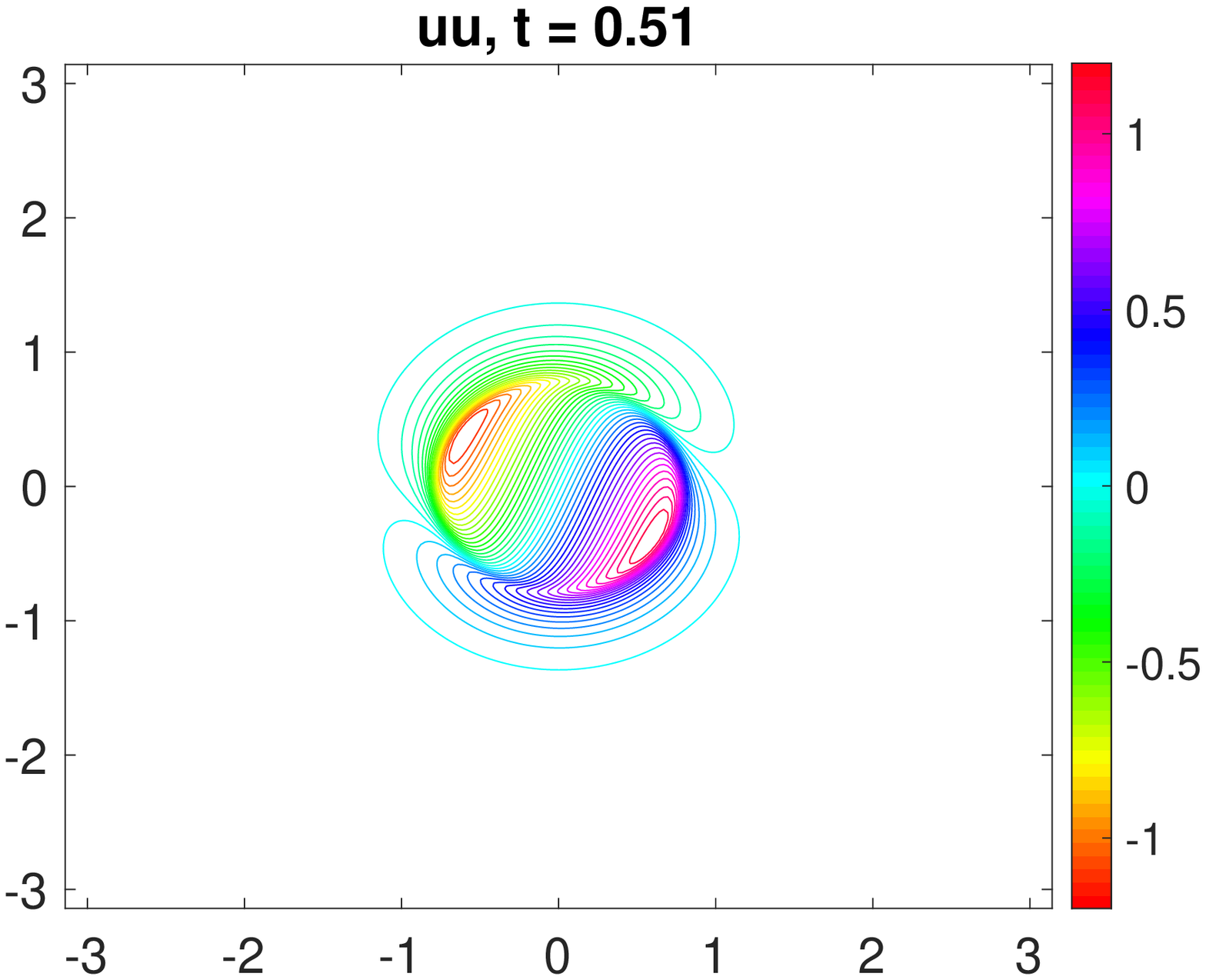}
		\end{subfigure}
		\begin{subfigure}{.32\textwidth}
			\centering
			\includegraphics[width=0.9\linewidth]{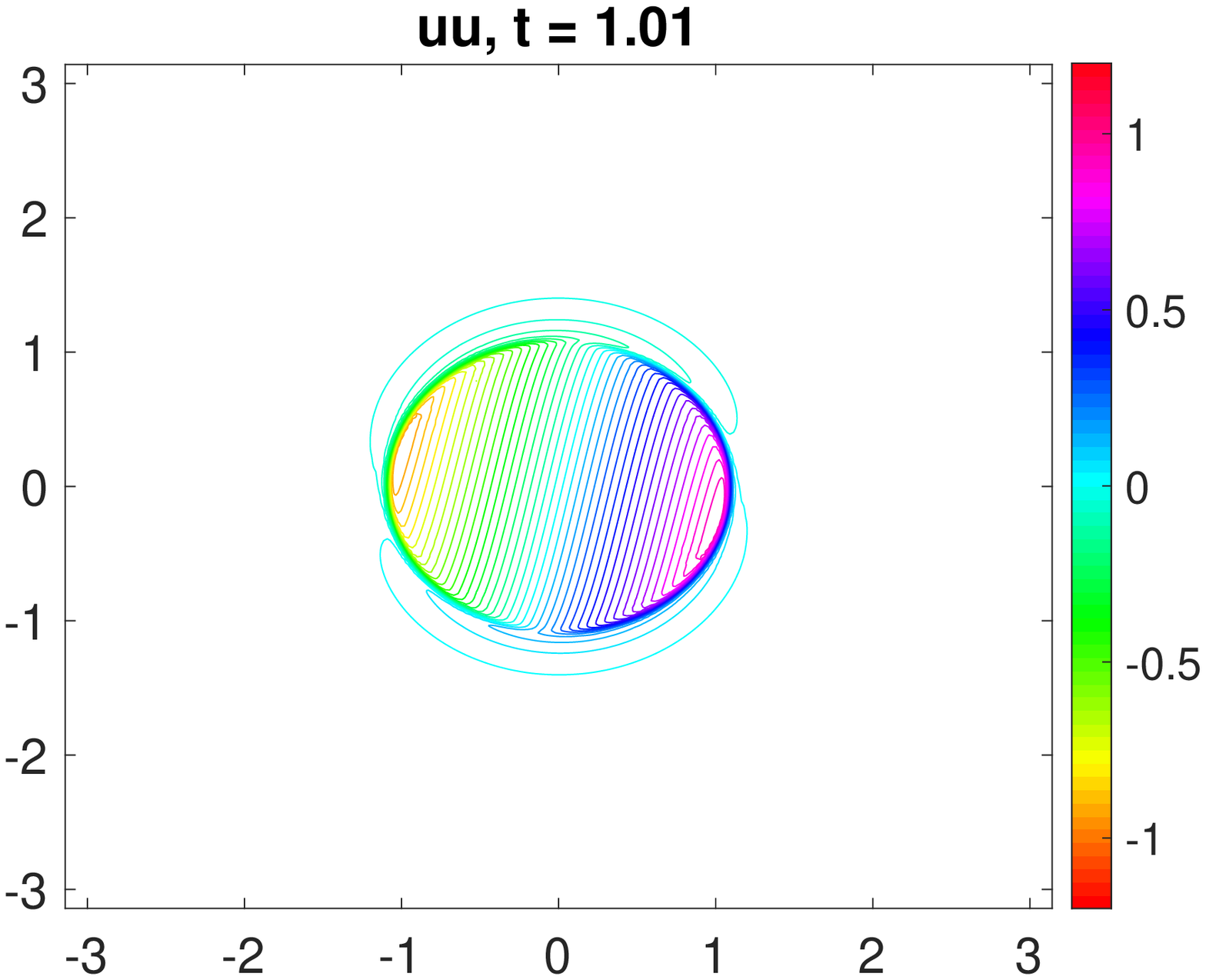}
		\end{subfigure}
		\\
		\begin{subfigure}{.32\textwidth}
			\centering
			\includegraphics[width=0.9\linewidth]{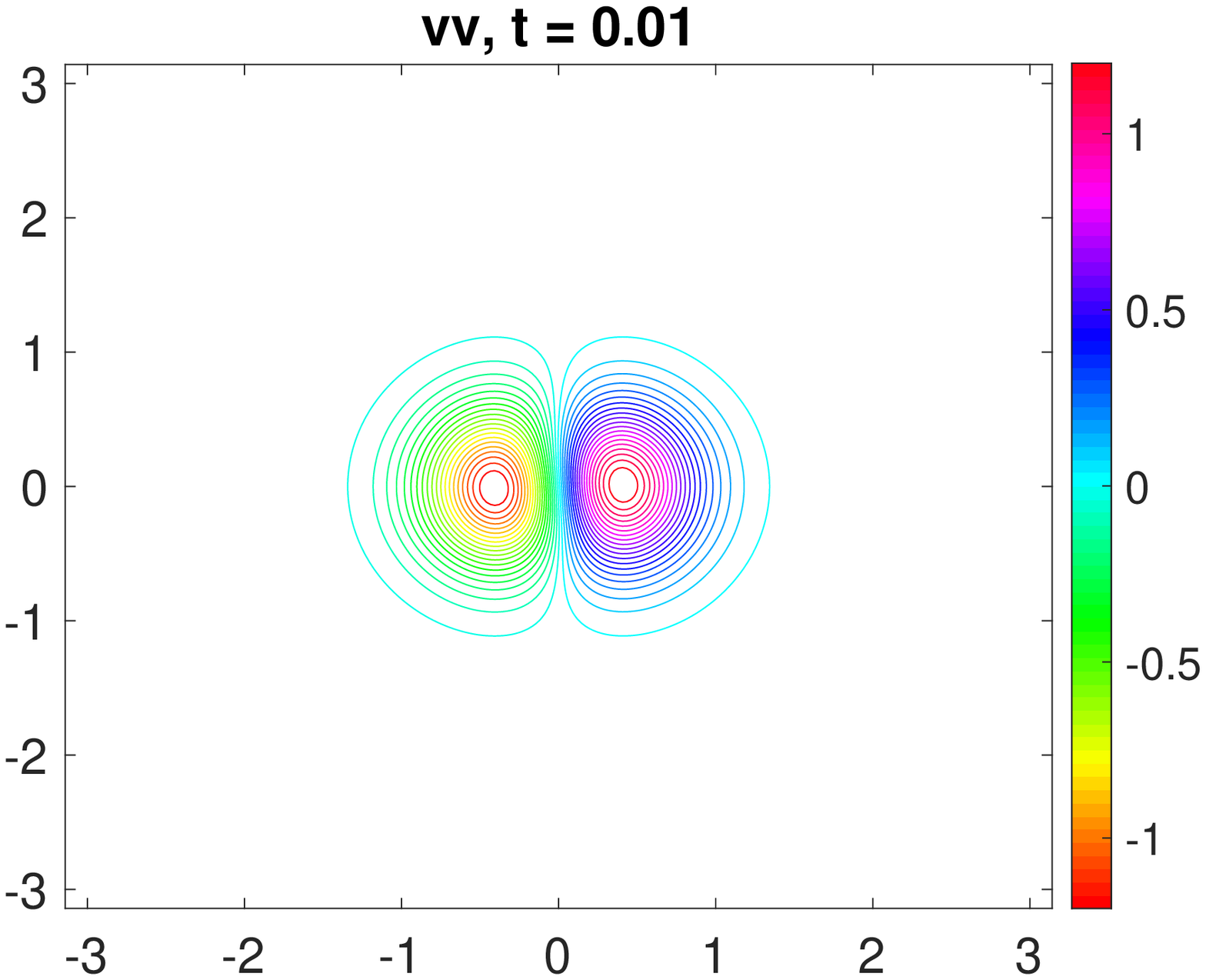}
		\end{subfigure}
		\begin{subfigure}{.32\textwidth}
			\centering
			\includegraphics[width=0.9\linewidth]{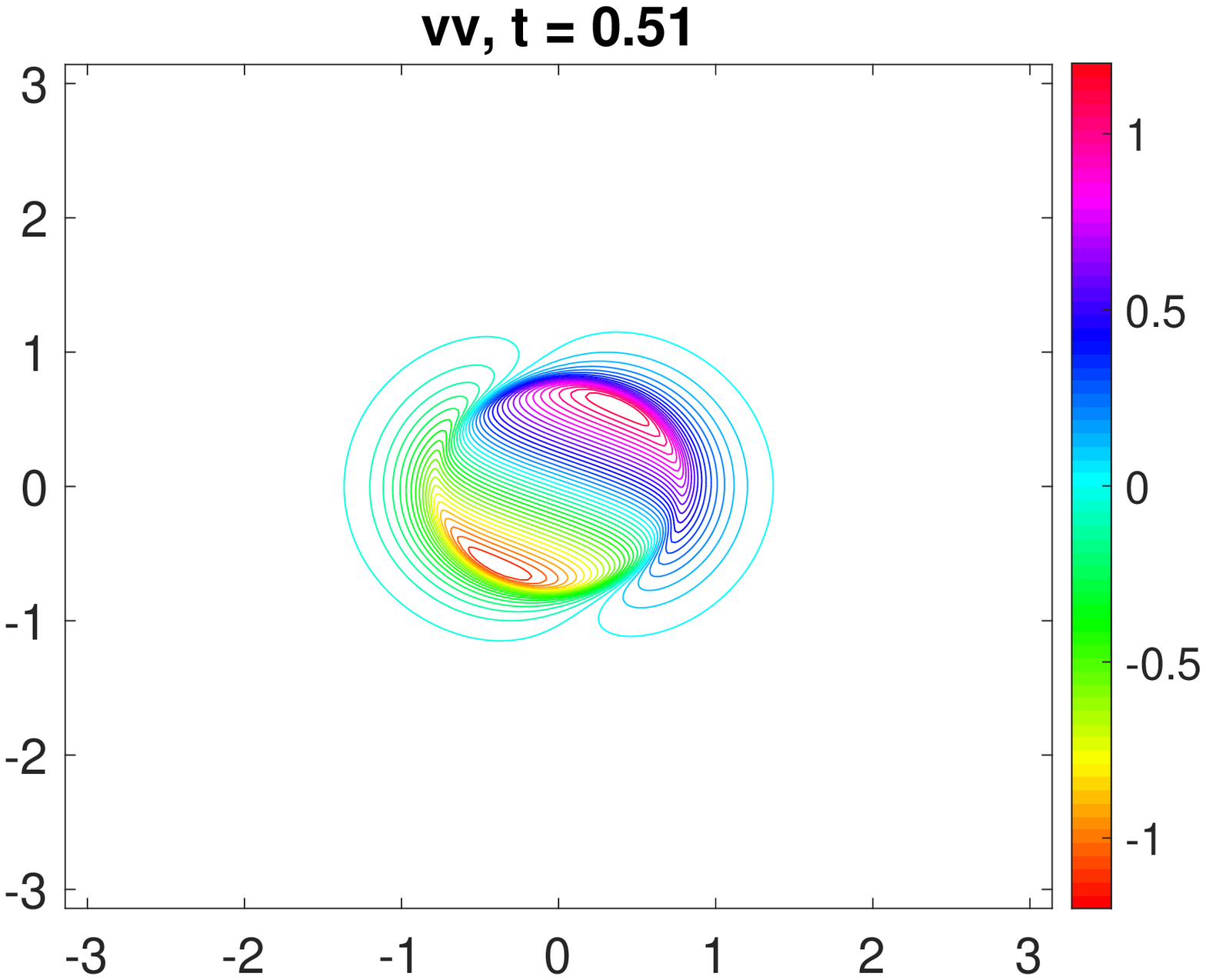}
		\end{subfigure}
		\begin{subfigure}{.32\textwidth}
			\centering
			\includegraphics[width=0.9\linewidth]{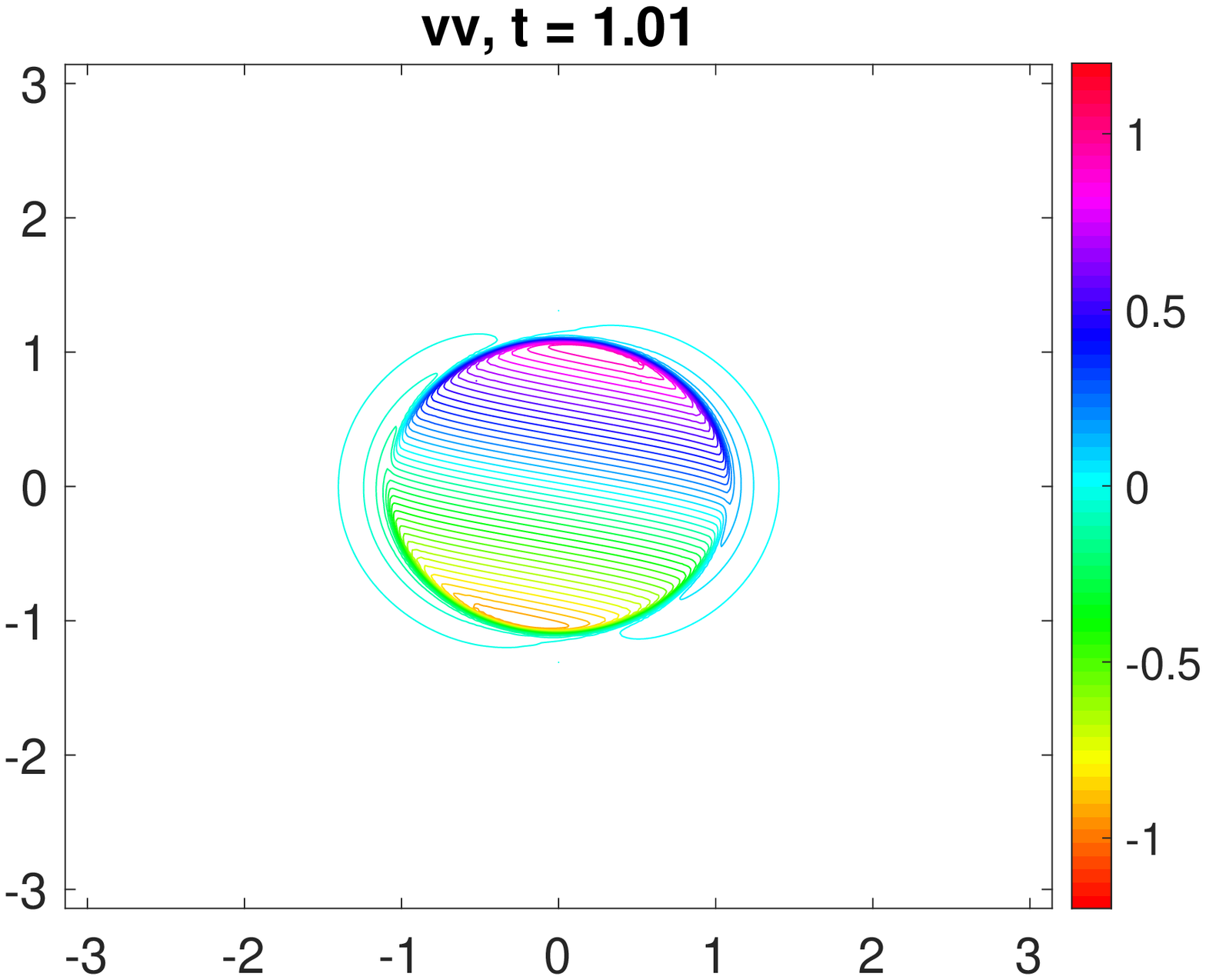}
		\end{subfigure}
		\caption{rotating flow with low viscosity}
		\label{fig:rotate}
	\end{figure}
	
	\subsection{Multiple filamentation}
	
	As a final example we simulate the regularized non-linear Schr\"{o}dinger equation \cite{Lushnikov:10}
	\begin{equation} \label{eq:rnls}
		i u_t + (1-ia \varepsilon) \Delta u + (1+i c \varepsilon)|u|^2u = i b \varepsilon u.
	\end{equation} 
	Here the terms with $a,b$ and $c$ corresponds to, linear wave-number dependent absorption,  linear gain, and two-photon absorption, respectively. This is a challenging equation to simulate as its solution approximates collapse of filaments, a process which requires accurate numerics in both space and time. When $\varepsilon = 0$ and the optical power is beyond a critical threshold localized filementation occurs. Each such collapse is well approximated by a self-similar radially symmetric solution 
	\[
	|u(r,t)| \approx R_0(\rho) / L(t), 
	\]
	with  $\rho = r / L(t)$ and
	\[
	L(t) \approx \sqrt(2\pi) \frac{\sqrt{t_0-t}}{\sqrt{\ln | \ln (t_0 - t) | }},
	\]
	where $t_0$ is the time of collapse. 
	
	
	
	We solve this problem on the domain $\Omega = [0,25.6]^2$ from $t=0$ to $t= 5s$. The domain is divided into $8\times 8$ leaf nodes with $p=13$, and timestep $\Delta t= 0.005$. The problem is equipped homogeneous Dirichlet boundary condition and with initial data
	\[
	u = 5\exp\left(-(x-12.8)^2 - (y-12.8)^2\right) \,.
	\]
	We plot the contour of the modulus of the solution at different time stamps $t=0,1.66,3.325$ and $t=4.995$ in Figure \ref{fig:filamentation}. We see that the solution has sharp derivatives and our method qualitatively captures the collapse behavior of filaments of the solution.
	\begin{figure}[htp]
		\begin{subfigure}{.45\textwidth}
			\centering
			\includegraphics[width=0.9\textwidth]{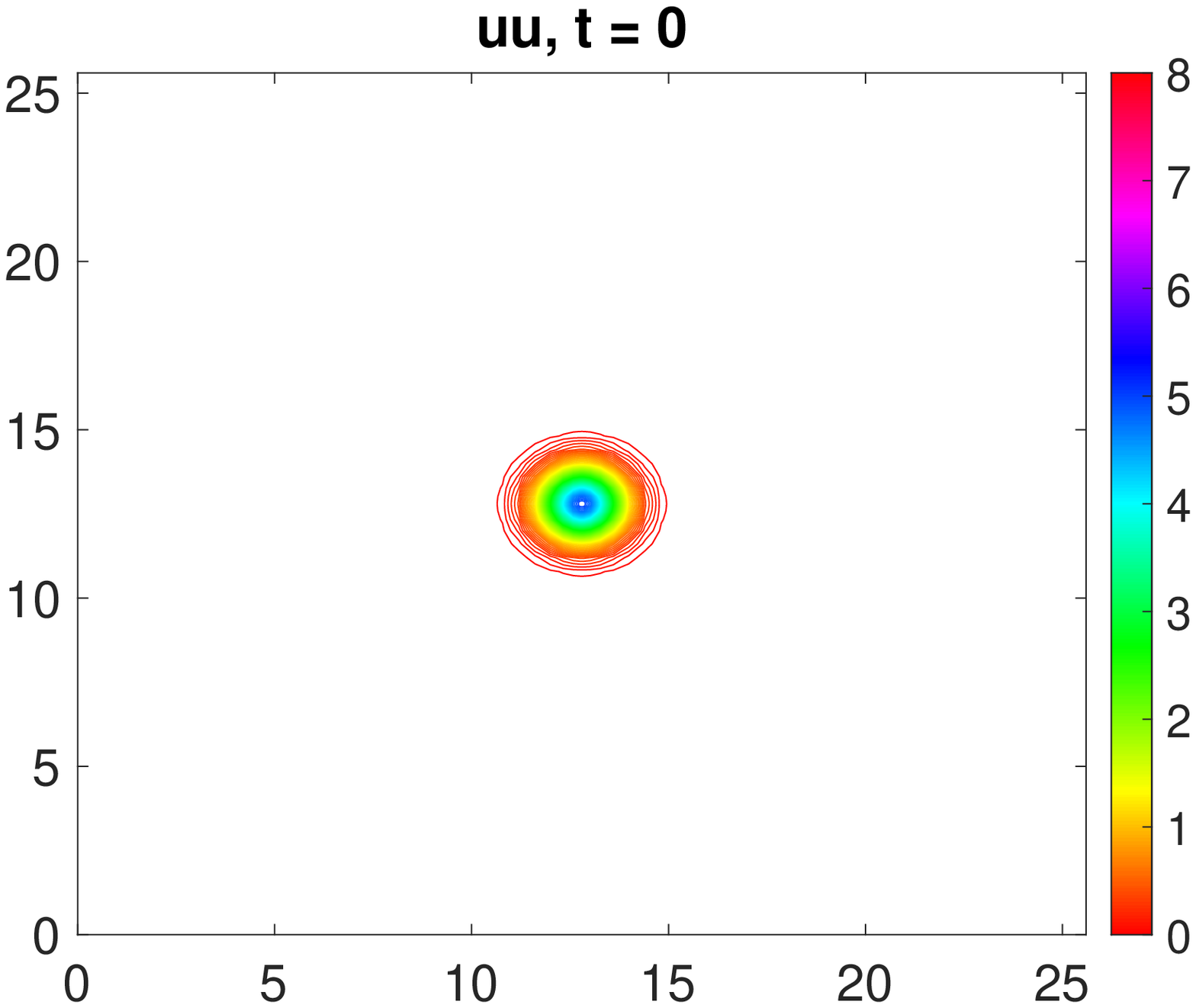}
		\end{subfigure}
		\begin{subfigure}{.45\textwidth}
			\centering
			\includegraphics[width=0.9\textwidth]{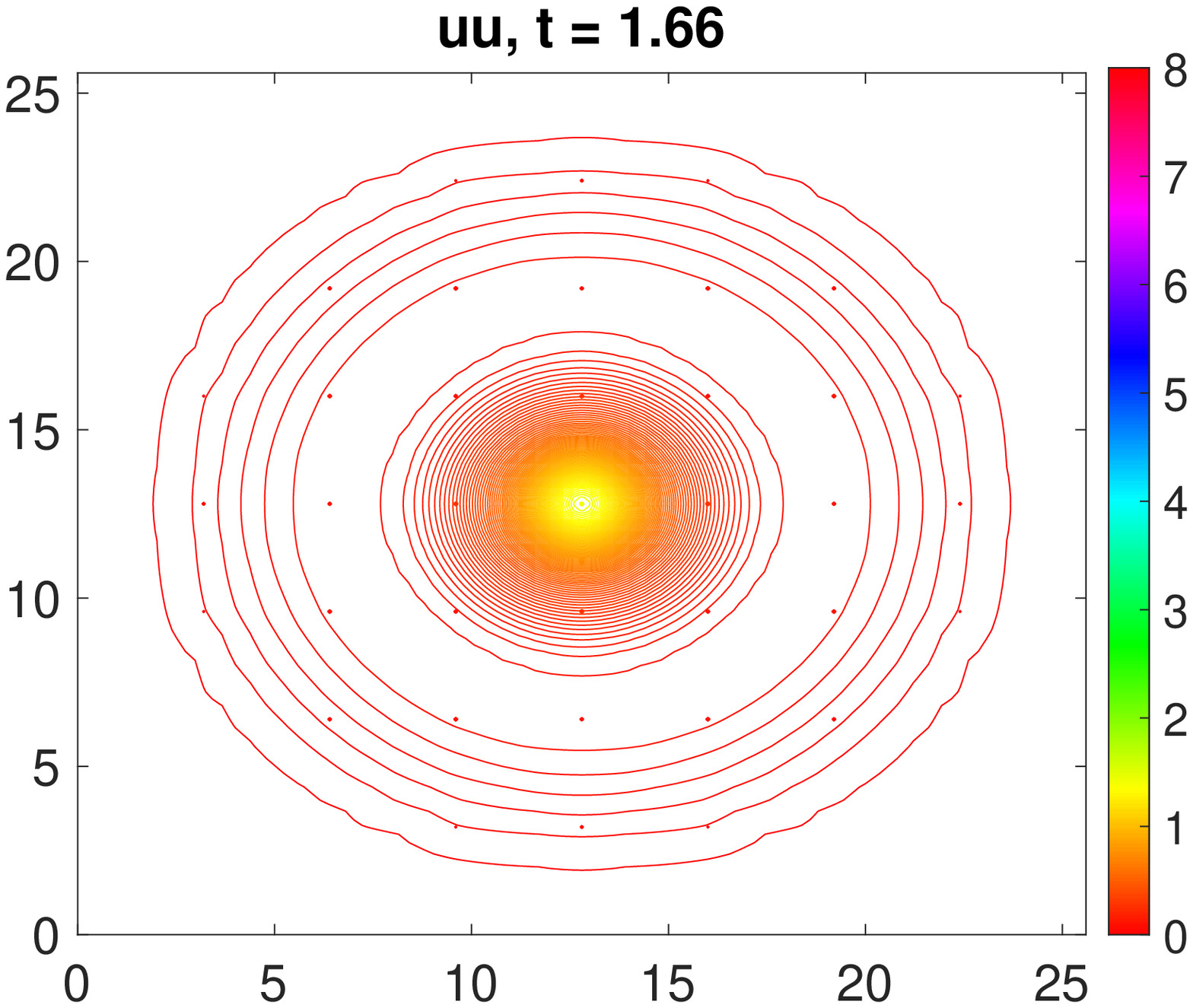}
		\end{subfigure}
		\\
		\begin{subfigure}{.45\textwidth}
			\centering
			\includegraphics[width=0.9\textwidth]{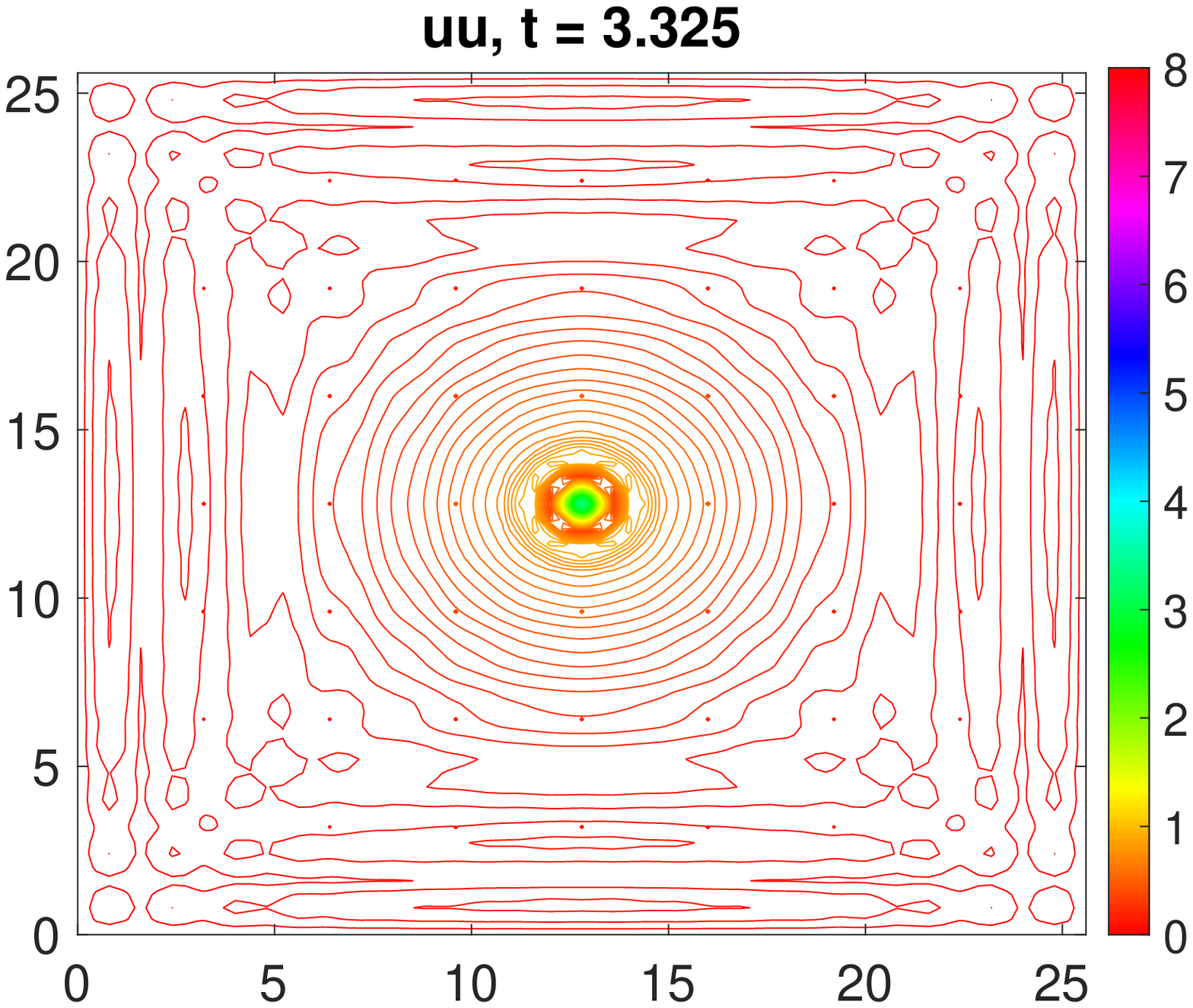}
		\end{subfigure}
		\begin{subfigure}{.45\textwidth}
			\centering
			\includegraphics[width=0.9\textwidth]{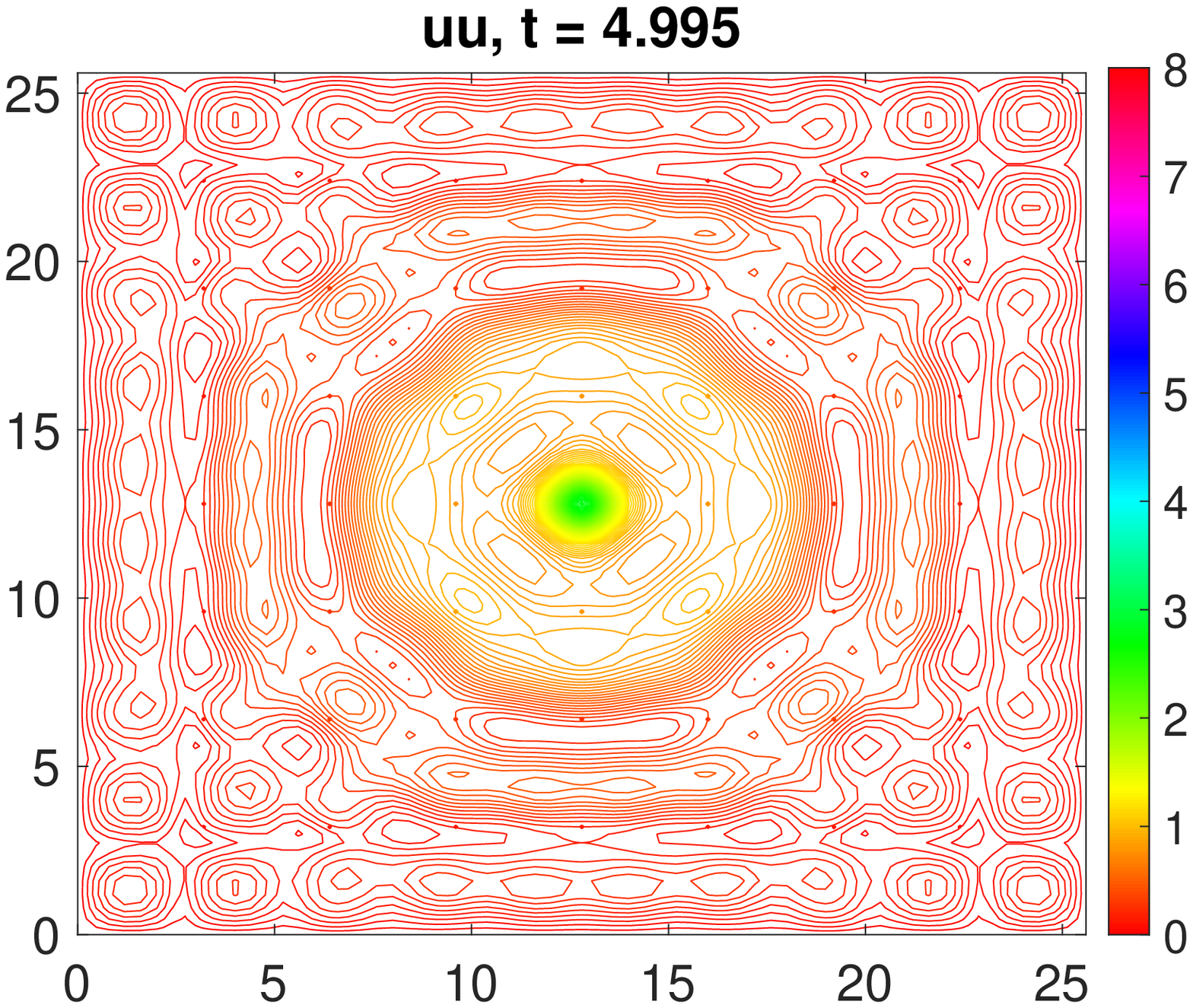}
		\end{subfigure}
		\caption{solution stamps of the regularized non-linear Schr\"odinger equation.}
		\label{fig:filamentation}
	\end{figure}

\section*{Acknowledgements}
This work was supported in part by NSF Grants DMS-2208164, DMS-2210286 (DA) and DMS-1952735, DMS-2012606 (PGM). The work of PGM was also supported by the Office of Naval Research (N00014-18-1-2354) and by the Department of Energy ASCR (DE-SC0022251). 

\section*{Conflict of Interest}
On behalf of all authors, the corresponding author states that there is no conflict of interest. All authors certify that they have no affiliations with or involvement in any organization or entity with any financial interest or non-financial interest in the subject matter or materials discussed in this manuscript. The authors have no financial or proprietary interests in any material discussed in this article.

\section*{Ethical Approval}
N/A

\section*{Informed Consent}
N/A

\bibliography{ref}

\end{document}